\documentclass[a4paper,12pt]{amsart}
\usepackage{amsfonts,amssymb,amscd}

\usepackage{epsf}
\usepackage{epsfig}
\usepackage{graphicx}

\newtheorem{teo}{Theorem}[section]
\newtheorem{prop}[teo]{Proposition}
\newtheorem{lema}[teo]{Lemma}
\newtheorem{obs}[teo]{Remark}
\newtheorem{defnc}[teo]{Definition}

\newcommand{\C}{{\mathbb C}}
\newcommand{\R}{{\mathbb R}}

\newcommand{\Z}{{\mathbb Z}}
\newcommand{\N}{{\mathbb N}}

\newcommand{\fol}{{\mathcal F}}

\newcommand{\calU}{{\mathcal U}}

\newcommand{\diff}{{{\rm Diff}\, ({\mathbb C}, 0)}}

\newcommand{\fdiff}{{\widehat{\rm Diff}\, ({\mathbb C}, 0)}}

\newcommand{\tilf}{{\widetilde{\mathcal F}}}

\newcommand{\om}{{\omega}}

\begin{document}

\title[Pseudogroups and Topology of Leaves]{Generic pseudogroups on $(\C ,0)$ and the topology of leaves}

\author{J.-F. Mattei, \hspace{0.4cm} J. C. Rebelo \hspace{0.2cm} \& \hspace{0.2cm} H. Reis}
\address{}
\thanks{}

\begin{abstract}
We show that generically a pseudogroup generated by holomorphic diffeomorphisms defined about $0 \in \C$ is free in the sense
of pseudogroups even if the class of conjugacy of the generators is fixed. This result has a number of consequences on the topology
of leaves for a (singular) holomorphic foliation defined on a neighborhood of an invariant curve. In particular in the classical and simplest case arising
from local foliations possessing a unique separatrix that is given by a cusp of the form $\{y^2-x^{2k+1}=0\}$, our results
allow us to settle the problem of showing that a generic foliation possesses
only countably many non-simply connected leaves and that this countable set is, indeed, infinite.
\end{abstract}

\maketitle

\section{Introduction}

This paper is motivated by several difficulties concerning to greater or lesser extent the ``topology of leaves'' that are encountered
in the study of
some well-known problems about (singular) holomorphic foliations. Yet most of these problems are essentially concerned with pseudogroups
generated by certain local holomorphic diffeomorphisms defined on a neighborhood of $0 \in \C$. For this reason, we shall begin our discussion
by stating our results in this context. First
consider the group $\diff$ of germs of diffeomorphisms and let this group be
equipped with the {\it analytic topology} introduced by Takens \cite{takens}. The precise definition of this topology will be given in Section~2, for the time being it suffices
to know that it possesses the Baire property. Here we remind the reader that a $G_{\delta}$-dense set, sometimes also called
a residual set, is nothing but a countable intersection of open and dense sets in a Baire space.
All the ``generic results'' stated in this part of the Introduction concern $G_{\delta}$-dense sets for this topology.
It is however worth pointing out that, once they are established, it is easy
to derive the ``generic'' character in other contexts by means of ideas similar to those employed in Section~2
(in particular in  suitable topologies involving the coefficients of their Taylor series at $0 \in \C$ and/or in the
sense of measure).
Next consider a $k$-tuple of local holomorphic diffeomorphisms $f_1, \ldots ,f_k$
fixing $0 \in \C$. The first theorem proved here states that we can perturb the $f_i$'s
{\it without altering their classes of holomorphic conjugacy in $\diff$}
so that the group they generate is the free product of the corresponding cyclic groups. The reason for keeping the conjugacy
class fixed will be clear
when we shall discuss the applications to singular foliations. For the time being, note only that this condition is
equivalent to preserving the order of a diffeomorphism
if this order is finite. Also, in the case of a diffeomorphism having a hyperbolic fixed point at $0 \in \C$, the
condition amounts to preserving the corresponding multiplier. Next let $(\diff)^k$ denote the product of $k$-copies
of $\diff$ endowed with the product analytic topology. Then we have:

\vspace{0.2cm}

\noindent {\bf Theorem A}. {\sl Suppose that $f_1, \ldots ,f_k$ as above are such that none of the $f_i$'s has finite order
(i.e. $f_i^j \neq {\rm id}$ for every $j \in \N^{\ast}$). Then there exists a $G_{\delta}$-dense set $\mathcal{V} \subset (\diff)^k$
such that, whenever $(h_1, \ldots ,h_k) \in \mathcal{V}$ the group generated by $h_1^{-1}\circ f_1 \circ h_1,
\ldots , h_k^{-1}\circ f_k \circ h_k$ induces a free group in $\diff$.
Furthermore, there is a fixed neighborhood $V$ of $0 \in \C$ such that, denoting by $\Gamma^h$ the pseudogroup of mappings
defined on $V$ by $h_1^{-1}\circ f_1 \circ h_1, \ldots , h_k^{-1}\circ f_k \circ h_k$, every element of $\Gamma^h$ coinciding with the
identity on a non-empty open set must coincide with the identity on all of its domain of definition in $V$.

Finally there are infinitely many points $p_1, p_2, \ldots$ contained in $V$ such that $p_i$ is a hyperbolic fixed point of some element
$g_i$ of $\Gamma^h$ and the orbits under $\Gamma^h$  of $p_i$ and $p_j$ are disjoint provided that $i\neq j$.}

\vspace{0.2cm}

The assumption that the order of each $f_i$ is infinite does not constitute a limitation of our methods.
In fact, let $G_i$ be the cyclic group generated
by $f_i$ (which may or may not be infinite). Theorem~B below is formally a generalization of Theorem~A whereas our methods actually
yield a unified proof of both statements.

\vspace{0.2cm}

\noindent {\bf Theorem B}. {\sl Let $f_1, \ldots ,f_k$ and $G_1, \ldots , G_k$ be as above.
Then there exists a $G_{\delta}$-dense set $\mathcal{V} \subset (\diff)^k$
such that, whenever $(h_1, \ldots ,h_k) \in \mathcal{V}$ the group generated by $h_1^{-1}\circ f_1 \circ h_1,
\ldots , h_k^{-1}\circ f_k \circ h_k$ induces a group in $\diff$ that is isomorphic to the free product
$$
G_1 \ast \cdots \ast G_k \, .
$$
\noindent Furthermore, there is a fixed neighborhood $V$ of $0 \in \C$ such that, denoting by $\Gamma^h$ the pseudogroup of mappings
defined on $V$ by $h_1^{-1}\circ f_1 \circ h_1, \ldots , h_k^{-1}\circ f_k \circ h_k$, every element of $\Gamma^h$ coinciding with the
identity on a non-empty open set must coincide with the identity on all of its domain of definition in $V$.

Finally there are infinitely many points $p_1, p_2, \ldots$ contained in $V$ such that $p_i$ is a hyperbolic fixed point of some element
$g_i$ of $\Gamma^h$ and the orbits under $\Gamma^h$  of $p_i$ and $p_j$ are disjoint provided that $i\neq j$.}

\vspace{0.2cm}

For the last part of the preceding statements, we note that the existence of a hyperbolic fixed point for the pseudogroup arising from
a ``non-solvable'' group has been known since \cite{scherba}, \cite{Frank1} whereas the existence of two or more hyperbolic
fixed points with disjoint orbits was not previously settled. It appears here as a consequence of our theory of
generic pseudogroups.

It is natural to expect the preceding statements to have consequences on the topology of the leaves of a foliation on a neighborhood
of an invariant curve or on a neighborhood of a singular point. In this paper we shall content ourselves of providing an answer to the
long-standing question of nilpotent singularities leaving a cusp of the form $\{y^2 + x^{2n+1} =0\}$ invariant.
By a small abuse of language, by a cusp of the form $\{y^2 + x^{2n+1} =0 \}$ it will always be meant
a (local) curve analytically equivalent to the cusp in question.
This choice will help us to explain most of the relevant ideas without making the discussion too technical.
To begin with consider a singular foliation $\fol$ defined about
$(0,0) \in \C^2$ by a nilpotent $1$-form $ydx +\cdots$ and possessing a cusp $\{y^2 + x^{2n+1} =0 \}$ as its unique separatrix.
It then follows that the linear part of $\fol$ at the origin is nilpotent (and non-zero). This much studied class
of singularities corresponds to Arnold's singularities of type $A^{2k+1}$.
Whereas several works were devoted to these nilpotent singularities, and in particular to the description
of suitable normal forms (cf.  \cite{rennes1}, \cite{poland} and references therein), the question about the topology
of most leaves for the ``generic foliation'' remained unsettled. Our Theorem~C below states that for a generic foliation
in the class $A^{2k+1}$ there are only a countable (infinite) number of leaves that are not simply connected.

Generic theorems about foliations, as it is the case of Theorem~C, are more commonly expressed in terms of the Krull
topology since the corresponding formulations automatically bear a meaning in terms of ``generic coefficients'' for
the differential equation in question. This explains why Theorem~C will be stated with respect
to the Krull topology rather than being ``parameterized'' by $G_{\delta}$-dense sets in $\diff$ (yet a formulation in terms
of $G_{\delta}$-dense sets is possible, cf. Section~5 for details).

Let $\om \in \Lambda_{(\C^2,0)}$ be a $1$-form with an isolated singularity at the origin and defining
a germ of a nilpotent foliation $\fol$ of type $A^{2k+1}$.

\vspace{0.1cm}

\noindent {\bf Theorem C (Cusps)}. {\sl For every (sufficiently large) $N \in \N$, there exists a $1$-form
$\om^{\prime} \in \Lambda_{(\C^2,0)}$ defining a germ of a foliation $\fol^{\prime}$ and satisfying the
following conditions:
\begin{itemize}
\item $J_0^N \om^{\prime} = J_0^N \om$ (i.e. the forms $\om, \, \om^{\prime}$ are tangent to order $N$ at the origin).

\item The foliations $\fol$ and $\fol^{\prime}$ have $S$ as a common separatrix.

\item There exists a fundamental system of open neighborhoods $\{U_n\}_{n \in \N}$ of $S$, inside a
closed ball $\bar{B}(0,R)$, such that for all $n \in \N$, the leaves of the restriction of $\fol^{\prime}$
to $U_n \setminus S$ are simply connected except for a countable number of them.
\end{itemize}
}

To apply Theorems~A and~B to the topology of leaves of foliations in more general settings is a quite subtle problem
for which the theory developed in \cite{marinmattei} becomes a powerful tool. Concerning our Theorem~C, a self-contained
proof is given in Section~5. This proof, however, amounts to applying the techniques of \cite{marinmattei} to an elementary case.
Another comment about Theorem~C is that, though it is naturally constructed, the systems of neighborhoods $U_n$
cannot be arbitrary. In fact, for an arbitrary neighborhood $U$ it may happen, for example, that intersections of leaves
with the boundary of $U$ create ``holes'' in the corresponding leaves which will no longer be simply connected.
In Section~5 a result slightly more accurate than Theorem~C will be stated. This section also
contains further information and details about these foliations.

To close this Introduction let us make some comments concerning the standard condition that we have considered
in Theorem~A and~B. Namely the fact that the analytic conjugacy class of the initial local diffeomorphisms is always
kept fixed. For this it is interesting to look at a foliation $\tilf$ defined on a neighborhood of a rational
curve $C$ (in turn embedded in some complex surface). The singularities of $\tilf$ in $C$ are denoted by $p_1, \ldots , p_k$.
It is then natural to consider perturbations of $\tilf$ satisfying our standard condition: the analytic class of the local holonomy
map $\sigma_k$ defined by a small loop around $p_k$ is fixed. Since the singularities are simple, this condition is equivalent
to saying that the {\it analytic types}\, of the singularities of $\tilf$ are fixed, as it follows from classical results. In turn,
if in addition these singularities belong to the Poincar\'e domain, then our context is equivalent to the context of
{\it isospectral deformations}\, i.e. deformations
preserving the eigenvalues of each singular point. For singularities belonging to the Siegel domain however our condition is far stronger than the
isospectral one and, in fact, it is expected to be
the natural ``good'' condition for developing a (global) moduli theory for holomorphic foliations.
Finally, when a singularity in the Siegel domain gives rise to a local holonomy of finite order, then the condition becomes
equivalent to deforming the foliation while keeping fixed the order of the local holonomy maps in question.
This last case is precisely the situation that emerges in the analysis
of singularities of type $A^{2k+1}$ and it will play a role in the proof of Theorem~C.

In any event, in a suitable sense, a generic foliation as above will still have all but countably many leaves simply connected.
This assertion may be justified by a construction similar to the construction carried out in Section~5. Alternatively we
can resort to more general results obtained in \cite{marinmattei}.
In particular the preceding theorems can be viewed as a step towards Anosov's conjecture stating a global result
about the existence of only countably many non-simply connected leaves for a generic foliation of the projective plane.

Concerning foliations on the projective plane leaving a projective line invariant, Il'yashenko and Pyartli \cite{ilya} proved that,
in the class of foliations with degree~$d$ of the projective plane leaving the line at infinity invariant, those for which
the holonomy group of (the regular part of) the line at infinity is free are generic. This very interesting
result has a different nature if compared to statements provided in this work and it deserves some comments. Whereas
Il'yashenko and Pyartli do not worry about how the singular points change in their considerations about ``generic foliations'',
one of the main differences between the two works stems from the fact that their theorem is stated for global foliations whose space of
parameters is far more restrictive and of finite dimension. Therefore their result does not apply in a singular context,
for example in the study of foliations leaving a cusp invariant, not only because the ``parameter space'' is totally different
but also because singularities are often ``deformed'' in their procedure. Similarly our construction does not apply in their
global context since it is unclear whether our ``perturbations'' can be realized within the natural parameter space
associated to (global) foliations of degree~$d$. Another issue that needs to be pointed out is that, unfortunately,
Il'yashenko and Pyartli's theorem works only at the ``infinitesimal'' level of the group of germs of diffeomorphisms
fixing $0 \in \C$. Due to the reasons explained above (cf. also Section~4), it therefore does not imply the existence of
simply connected leaves (apart from a countable set) in a fixed neighborhood of the line at infinity. Given the interest
of this type of question, we may wonder whether a suitable blend of ideas in
both papers may lead us to fill in some of the gaps mentioned above and provide further insight into
the general case of Anosov's conjecture.

Finally a word about the structure of the paper. In Section~2 we provide the definition of the analytic topology in the
context adapted to our needs. For the convenience of the reader a proof that this topology possesses the Baire property
was also included. We then provide some general elementary statements comparing formal ``perturbations'' with analytic perturbations
that might also be useful in other similar problems. In Section~3 we introduce our basic technique of perturbation relying on
Riemann maps and on Caratheodory theorem. This section ends with the proof of the ``infinitesimal versions'' of Theorems~A
and~B. In other words we find that we can obtain ``free subgroups'' contained in the group of {\it germs}\, $\diff$. In Section~4
we elaborate much further on the preceding material so as to be able to prove the full statements of Theorems~A and~B along with
some simple variants needed for Section~5. In all these sections we deal exclusively with the case where the corresponding groups
are generated by only two local diffeomorphisms (i.e. $k=2$ in the statements of Theorems~A and~B). This serves only to abridge
notations since the general case does not offer any additional difficulty.
Finally in Section~5 some implications of these theorems to foliations are discussed and Theorem~C is proved.

\section{Formal series and convergent series}

In the sequel let $\diff$ stands for the group of local holomorphic
diffeomorphisms fixing $0 \in \C$. The group of formal series
$\sum_{i=1}^{\infty} c_i x^i$, $c_i \in \C$, is going to be denoted
by $\fdiff$. There is an obvious injection of $\diff$ in $\fdiff$ which
associates to an element of $\diff$ its Taylor series about $0 \in \C$.

Let us begin by defining the so called {\it analytic topology}\,
(or $C^{\omega}$-topology) in $\diff$. To the best of our knowledge this
type of topology was first considered by Takens \cite{takens} who also
observed that it possesses the Baire property.

Given $r >0$, let $B(r) \subset \C$ denote the open disc of radius $r$ about
$0 \in \C$. Consider a holomorphic function $h$ defined around $0 \in \C$
and taking values in $\C$. If $h$ possesses a holomorphic extension (still
denoted by $h$) to $B(r)$, then set $\Vert h \Vert_r = \sup_{z \in B(r)}
\vert h(z) \vert$. Otherwise we pose $\Vert h \Vert_r = + \infty$.

Next for $r, \varepsilon >0$ and $f \in \diff$ chosen, let
$(f + \calU^{\varepsilon}_r) \subseteq \diff$ be the set defined by
$$
(f + \calU^{\varepsilon}_r) = \{ g \in \diff \; \; ; \; \; \Vert
g-f \Vert_r < \varepsilon \} \, ,
$$
where $(g-f)$ is interpreted simply as a holomorphic function that need
not be a local diffeomorphism at $0 \in \C$. The {\it analytic topology}\, on
$\diff$ can now be defined as being the one generated by the sets
$(f + \calU^{\varepsilon}_r)$. In other words, the sets $(f + \calU^{\varepsilon}_r)$
form a {\it basis of open sets} for the analytic topology. An immediate
consequence of this definition is that a sequence $\{ f_i \}_{i \in \C}
\subset \diff$ is convergent in the analytic topology if and only if,
to every pair $r, \varepsilon >0$, there corresponds $N \in \N$ such that
$\Vert f_i - f_j \Vert_r < \varepsilon$ whenever both $i, \, j$ are greater
than $N$.

\begin{obs}
\label{none}
{\rm Fixed $k \in \N^{\ast}$, the $k$-jet projection from $\diff$ to $\C^k$ defined by
$$
g \longmapsto (g^{(1)} (0) , g^{(2)} (0), \ldots , g^{(k)} (0))
$$
is continuous in the analytic topology as an easy consequence of Cauchy estimates. Conversely, this projection admits a natural section $T$
(Taylor polynomial) defined by
$$
(a_1 , \ldots , a_k) \longmapsto a_1 z + \frac{a_2}{2} z^2 + \cdots +   \frac{a_k}{k!} z^k
$$
which is also continuous for the topologies in question. This simple remark will often be used in the course of this work.}
\end{obs}

Following \cite{takens} let us show that the set $\diff$, endowed with the analytic
topology, is a Baire space.

\begin{lema}
\label{takens}
The group $\diff$ equipped with the analytic topology is a Baire space,
namely the countable intersection of open dense sets is dense.
\end{lema}

\begin{proof}
For every $i \in \N$, let $U_i$ denote an open and
dense subset of $\diff$. Given a nonempty open set $V \subseteq \diff$, we must
check that $V \cap (\bigcap_{i \in \N} U_i) \neq \emptyset$. Clearly it
suffices to show this when $V$ is an open set of the form
$V = (f + \calU^{\varepsilon}_r)$, for certain $r, \varepsilon >0$ and
$f \in \diff$.

To do this, let us construct a sequence of triplets $(f_i, r_i ,\varepsilon_i)$
verifying the following conditions:
\begin{enumerate}

\item $(f_0, r_0 ,\varepsilon_0) = (f,r,\varepsilon)$.

\item $f_i \in (f_{i-1} + \calU_{r_{i-1}}^{\varepsilon_{i-1}/2}) \cap U_i$,
($\varepsilon_0 = \varepsilon$).

\item $(f_i + \calU_{r_i}^{\varepsilon_i}) \subseteq
(f_{i-1} + \calU_{r_{i-1}}^{\varepsilon_{i-1}/2}) \cap U_i$.

\item $\{ \varepsilon_i\}$ decreases monotonically to {\it zero}\,
while $\{ r_i \}$ increases monotonically to $+\infty$.

\end{enumerate}

The existence of a sequence as above is clear and, thanks to the convergence
criterion stated above, it follows that $\{ f_i\}$ converges in the analytic
topology to a certain $f_{\infty} \in \diff$. We claim that
$f_{\infty}$ belongs to $V \cap (\bigcap_{i \in \N} U_i)$. Indeed, chosen
$i_0 \in \N$, the diffeomorphism $f_j$ belongs to
$(f_{i_0} + \calU_{r_{i_0}}^{\varepsilon_{i_0}/2})$ for all $j > i_0$. Therefore
the limit $f_{\infty}$ belongs to $(f_{i_0} + \calU_{r_{i_0}}^{\varepsilon_{i_0}/2})
\subseteq V \cap U_{i_0}$. Since $i_0 \in \N$ was arbitrarily chosen,
the statement follows at once.
\end{proof}

\begin{obs}
\label{topologicalgroup}
{\rm Note that $\diff$ endowed with the analytic topology is not a
topological group. This happens because the mapping from $\diff$
to $\diff$ that associates to a chosen $h \in \diff$ the element
$f \circ h$, where $f \in \diff$ is fixed, is not continuous in
general. In fact, a sequence $\{ h_i \} \subset \diff$ converges
in the analytic topology to $h$ if and only if $h_i =h+r_i$ where
$r_i \in \calU_r^{\varepsilon}$ for every fixed $r, \varepsilon > 0$
and sufficiently large $i$. However, if $f$ has a bounded domain of
definition, this does not guarantee that $f \circ (h+r_i) -f\circ h$
admits a holomorphic extension to arbitrarily large discs.
This remark was once communicated to the second author by L. Lempert
to whom we wish to thank.}
\end{obs}

Let us now turn to the central point of this section. Let $f, g \in \diff$
(resp. $\hat{f}, \hat{g} \in \fdiff$) be two holomorphic diffeomorphisms
(resp. formal series) fixing the origin and assume that both $f, g$ (resp.
$\hat{f}, \hat{g}$) are distinct from the identity map. Denote by $G$ (resp.
$\hat{G}$) the subgroup of $\diff$ (resp. $\fdiff$) generated by $f, \, g$
(resp. by $\hat{f}, \, \hat{g}$). Naturally every reduced word $W(a,b)$ in the letters $a, a^{-1}, b, b^{-1}$ represents
an element of $G$ (resp. $\hat{G}$) by means of the substitutions $a = f$ and $b = g$ (resp $a = \hat{f}$ and $b = \hat{g}$).
It is understood that the group law of $G$ (resp. $\hat{G}$) becomes identified with the composition
of the holomorphic diffeomorphisms (resp. formal series).

The problem we are considering is the following one. Assume we are given two
holomorphic diffeomorphisms $f, g$ (resp. formal series $\hat{f}, \hat{g}$)
that satisfy a non-trivial relation in the sense that there exists a nontrivial
reduced word $W(a,b)$ such that $W(f,g) = {\rm id}$ (resp. $W(\hat{f},\hat{g}) =
{\rm id}$). We want to know if there exists $h \in \diff$ (resp. $\hat{h} \in
\fdiff$) very close to the identity map such that $f, h^{-1} \circ g \circ h$
does not satisfy any non-trivial relation.

\begin{defnc}
\label{Urelation}
Consider a pair $f, g$ of holomorphic diffeomorphisms in $\diff$ and consider
also a reduced word $W(a, b)$ such that $W (f, g) = {\rm id}$ in $\diff$. The word
$W(a, b)$ is said to be a universal relation if for every $\hat{h} \in \fdiff$ we
still have $W (f, \hat{h}^{-1} \circ g \circ \hat{h}) = {\rm id}$.
\end{defnc}

In the above statement $W(f,g)$ stands for the element of $\diff$ obtained by the
substitution $a = f$, $a^{-1} = f^{-1}$, $b = g$ and $b^{-1} = g^{-1}$ on $W(a,b)$
and by considering the group law as being the composition of formal series. An
analogous remark applies to $W (f, \hat{h}^{-1} \circ g \circ \hat{h})$.

We can finally state the main result of this section.

\begin{teo}
\label{AlgebraicTh}
Let $f, g \in \diff$ be two non-trivial holomorphic diffeomorphisms fixing the origin
and let $W(a,b)$ be a non-trivial reduced word in $f,g$ and their inverses. If $W (a,b)$
is not a universal relation, i.e. if there exists $\hat{h} \in \fdiff$ such that
$W (f, \hat{h}^{-1} \circ g \circ \hat{h}) \ne {\rm id}$ then there also exists
$h \in \diff $ such that $W (f, h^{-1} \circ g \circ h) \ne {\rm id}$. In fact the
set of elements $h \in \diff$ such that $W (f, h^{-1} \circ g \circ h) \ne {\rm id}$
is open and dense for the analytic topology.
\end{teo}

The theorem above allows us to prove also the following result:

\begin{teo}\label{representation}
Suppose we are given $f, \, g$ satisfying no universal relation. Then there exists a
set $\calU \subseteq \diff$, dense for the analytic topology, such that if $h \in \calU$
then $f, \, h^{-1} \circ g \circ h$ generates a free subgroup of $\diff$.
\end{teo}

Note that the theorem above assumes that ``relations can be broken" in the category of formal
series to conclude the existence of no relations in the category of convergent power series
in $\diff$.

\begin{proof}[Proof of Theorem~\ref{representation}]
Consider a non-trivial reduced word in $f,g$ and their inverses.
By assumption there is no universal relation $W$ for the pair $f, \, g$. This means that for every
word $W$ there exists $\hat{h}_W \in \fdiff$ such that $W(f, \hat{h}_W \circ g \circ \hat{h_W}^{-1})
\ne {\rm id}$. According to Theorem~\ref{AlgebraicTh} there exists an open dense
set $\calU_W \subseteq \diff$ in the analytic topology whose elements $h$ satisfy $W(f, h \circ g \circ h^{-1})
\ne {\rm id}$. Since there are only countably many words $W$ and since the analytic topology is a Baire
space, the set
\[
\calU = \bigcap_{W \in G} \calU_W
\]
is dense in $\diff$. It is immediate to check that the elements $h$ in $\calU$ have the property that
$f, \, h^{-1} \circ g \circ h$ generates a free subgroup of $\diff$.
\end{proof}

To prove Theorem~\ref{AlgebraicTh} let us begin by considering holomorphic diffeomorphisms
$f,g \in \diff$ such that $W(f,g) = {\rm id}$ for a fixed non-trivial reduced word $W$.
Consider also a formal series $\hat{h} \in \fdiff$ such that $W (f, \hat{g}) \ne {\rm id}$,
where $\hat{g} = \hat{h}^{-1} \circ g \circ \hat{h} \in \fdiff$. Assume that $W (f, \hat{g})
= \sum_{i \geq 1} a_i z^i$. Since $W (f, \hat{g}) \ne {\rm id}$ then there exists $K \geq 2$
such that $a_K \ne 0$. Denote by $\hat{h}_{K}$ the polynomial of degree $K$ obtained by
truncating $\hat{h}$ in the obvious way and let $\hat{g}_K = \hat{h}^{-1}_K \circ g \circ
\hat{h}_K$. Since $W(f, \hat{g})$ coincides with $W(f, \hat{g}_K)$ to the order $K$, it
immediately follows that $W (f, \hat{g}_K) \neq {\rm id}$ as well.

Fix $K$ as above and let us identify the coefficients of $\hat{h}_K$ to a point $P_K$ in $\C^K$.
Next note that the coefficients of $\hat{h}^{-1}$ are entirely determined by those of $\hat{h}$
through explicit algebraic formulas. In particular the coefficients of $\hat{h}_K^{-1}$ are algebraic
functions of the coefficients of $\hat{h}_K$ and so are the coefficients of $\hat{h}^{-1}_K \circ g
\circ \hat{h}_K$. Note however that $\hat{h}_K^{-1}$ is not a polynomial in general. All this can
be summarized by saying that, with the preceding identifications, we have a natural affine algebraic
map
$$
J^r \circ W\, : \; \C^K \longrightarrow \C^r \, ,
$$
for a suitable $r \in \N$, whose value at a given point consists of the $r$-tuple formed by
the coefficients of the resulting word until degree $r$. In other words, for a chosen $r$,
$J^r \circ W (P_K)$ is the $r$-jet of the formal series $W (f, \hat{g}_K)$. The preimage of
$(1,0, \ldots ,0) \in \C^r$ is by construction a Zariski-closed affine subset of $\C^{K}$.
This is a properly contained subset for $r \geq K$ since $W (f, \hat{g}_K) \neq {\rm id}$ or, more precisely,
since the coefficient of $z^K$ in the series expansion of $W (f, \hat{g}_K)$ is different from {\it zero}.
Therefore we have proved:

\begin{lema}
\label{zariski1}
With the preceding identifications, there is a non-empty Zariski open subset ${\mathcal O}_W$
of $\C^{K}$ whose points, thought of as polynomials $\hat{g}_K$, satisfy $J^r \circ W (P_K)
\neq {\rm id}$.\qed
\end{lema}

Note that the lemma above holds for every sufficiently large value of $K$ and $r (\geq K)$.
Furthermore for $r=K$ if $W(f, \hat{g}) = {\rm id}$ so does $J^r \circ W(f, \hat{g}_K)$.
We are finally ready to prove Theorem~\ref{AlgebraicTh}.

\begin{proof}[Proof of Theorem~\ref{AlgebraicTh}]
Recall that $W$ is a non-trivial reduced word that does not constitute a universal relation.
Therefore there is $\hat{h} \in \fdiff$ such that $W(f, \hat{h} \circ g \circ \hat{h}^{-1})
\ne {\rm id}$. In turn Lemma~\ref{zariski1} asserts that for sufficiently large $K$ and $r \geq K$
there exists a non-empty Zariski open subset of $\C^K$, denoted by ${\mathcal O}_W$, whose points,
thought of as polynomials, satisfy $J^r \circ W(P_K) \ne 0$. Being Zariski open and non-empty,
${\mathcal O}_W$ is in particular dense for the ordinary topology of $\C^K$. We can take $r = K$.

Let $\calU_W \subset \diff$ be the open set in the analytic topology of elements $h$ satisfying
$W(f, h \circ g \circ h^{-1}) \neq {\rm id}$. Our purpose is to show that the set $\calU_W$ is not
only non-empty but also dense in the analytic topology of $\diff$.

Fix an element $h$ of $\diff$ and let us assume that $W(f, h \circ g \circ h^{-1}) = {\rm id}$. We
shall prove the existence of elements $\tilde{h}$ arbitrarily close to $h$ in the analytic topology
and satisfying $W(f, \tilde{h} \circ g \circ \tilde{h}^{-1}) \ne {\rm id}$. Let $K$ be as above
and denote by $g_h$ (resp. $g_{h,K}$) the element $h \circ g \circ h^{-1}$ (resp. $h_K \circ g
\circ h^{-1}_K$). Then the coefficients of $g_{h,K}$ depend only on the first $K$
coefficients of $h$. Moreover this dependence is given by algebraic functions. Assume that
\[
h = \sum_{k \geq 1} a_i z^i = a_1 x + \cdots + a_K x^K + \cdots
\]
is the Taylor's expansion of $h$ at the origin. Since $W(f, g_h) = {\rm id}$ and since $W(f, g_{h,K})$
coincides with $W(f, g_h) = {\rm id}$ to order $K$, it follows that $J^K \circ W(f, g_{h,K}) = {\rm id}$
as well. This means that $(a_1, \ldots, a_K) \in \C^K \setminus {\mathcal O}_W$. However ${\mathcal O}_W$
is a non-empty Zariski open set of $\C^K$ so that there exists $(\tilde{a}_1, \ldots, \tilde{a}_K) \in {\mathcal O}_W$
that is arbitrarily close to $(a_1, \ldots, a_K)$. Now let $\tilde{h}$ be the formal series given by
\[
\tilde{h} = \sum_{i \geq 1} b_i z^i
\]
where $b_i = \tilde{a}_i$ if $1 \leq i \leq K$ and $b_i = a_i$ for $i > K$. We claim the following holds:
\begin{enumerate}
\item $\tilde{h}$ is convergent so that it defines an element of $\diff$
\item $W(f, g_{\tilde{h}}) \ne {\rm id}$
\item $\tilde{h}$ can be made arbitrarily close to $h$ in the analytic topology. In other words, given a
neighborhood of $h$, there exists $\tilde{h} \in \diff$ lying in this neighborhood and fulfilling conditions
(1) and (2) above.
\end{enumerate}

The fact that $\tilde{h}$ is convergent follows from observing that the power series of $\tilde{h}$ coincides with
the power series of $h$ up to a finite number of terms. Condition (2) is an immediate consequence of the
fact that $W(f, g_{\tilde{h}})$ coincides with $W(f, g_{\tilde{h}})$ to order $K$ where $h_K$ is obtained
from the point $(\tilde{a}_1, \ldots, \tilde{a}_K) \in {\mathcal O}_W$.

Finally, to check condition (3) let a neighborhood $U$ of $h$ in the analytic topology be fixed. Without
loss of generality we can assume that $U = h + \calU^{\varepsilon}_r$ for some $\varepsilon , r > 0$.
Next note that $h - \tilde{h}$ is a polynomial of degree $K$. More precisely we have
\[
(h - \tilde{h})(z) = (a_1 - b_1)z + \cdots + (a_K - b_K) z^K \, .
\]
with $a_i$ and $b_i$ as above. As a polynomial, it admits a holomorphic extension to $B(r)$. In fact
$h - \tilde{h}$ it is defined on all of $\C$. Since ${\mathcal O}_W$ is dense in the ordinary topology of $\C^K$
we can choose $b_i$ arbitrarily close to $a_i$. In particular, we can choose $b_i$ such that
$\Vert h - \tilde{h} \Vert_r < \varepsilon$. This concludes the proof of the theorem.
\end{proof}

\section{Destroying relations}

From now on we shall limit ourselves to proving our main theorems in the case of two (germs of) diffeomorphisms.  This will help us
to avoid cumbersome notations and make the proof more transparent. The required adaptations to deal with the general case are
more than straightforward.

The purpose of this section is to guarantee the non-existence of universal relations under
the preceding conditions. The proof is based on the study of the convergence of a suitable sequence of
Riemann mappings. The non-existence of universal relations will quickly lead to the proofs of our main
results. It should also be pointed out that we shall ``destroy non-trivial relations" by considering
only convergent power series, though it would be enough to destroy them in the category of formal series
as shown in the previous section. Yet we cannot dispense with the discussion carried out in Section~2
since we must show that the set of diffeomorphisms ``$h$" that ``breaks a given relation" is open and dense
in a suitable topology possessing Baire property.

In any event the main result of this section is as follows:

\begin{teo}\label{existence}
With the notations of Section~2, there is no universal relation consisting of a reduced word that is not of
the form $f^k$ or $g^l$ with $k,l \in \Z$.
\end{teo}

Before proving Theorem~\ref{existence} let us explain the role that Riemann mappings are going to play
in the proof. To do this let us give a brief sketch of the proof of Theorem~\ref{existence}.
To begin with recall that a holomorphic map is said to be {\it univalent}\, on a domain $U$ if it is
one-to-one on $U$, i.e. if it provides a holomorphic diffeomorphism from $U$ onto its image.

Now fix a pair of elements $f, \, g \in \diff$ and a reduced non-trivial word $W (a,b)$ on $a,b$ and their inverses. 
Suppose that $W(f,g) = {\rm id}$ with the natural substitutions $a=f$, $b=g$. Then for every $z$ in a small neighborhood of the origin we
have that $W(f,g)(z) = z$, where $W(f,g)$ is naturally identified to a local diffeomorphism fixing $0 \in \C$.
We need to find an element $\hat{h}_W \in \fdiff$ such that the formal series at the origin arising from
the formal composition $W(f,\hat{h}^{-1} \circ g \circ \hat{h})$ is not reduced to $z$. In fact, we shall
construct an element $h$ belonging to $\diff$, i.e. a convergent power series, satisfying this requirement.
The idea to obtain the desired element $h \in \diff$ is as follows. Let $z \neq 0$ be a point sufficiently
close to $0 \in \C$ so that its itinerary $z_0, z_1, \ldots, z_k$ with respect to the word $W$ is well-defined.
By the {\it itinerary}\, of $z$ it is meant the sequence of points obtained in the following way: set
$W(f,g) = \circ_{i=1}^k a_i$ where $a_i \in \{f^k, (f^{-1})^k, g^j, (g^{-1})^j, j \in \N\}$, where
$f^j, (f^{-1})^j$ (resp. $g^j, (g^{-1})^j$) stands for $f \circ \cdots \circ f$, $f^{-1} \circ \cdots
\circ f^{-1}$ (resp. $g \circ \cdots \circ g$, $g^{-1} \circ \cdots \circ g^{-1}$) $j$ times. Note that
every subsequence of consecutive $f$'s and/or of $f^{-1}$'s (resp. $g$'s and/or of $g^{-1}$'s) are grouped
to be considered as a unique element. The reason to do this will become clear below. The itinerary of $z$
is then given by the sequence $z_0 = z$, $z_i = a_i(z_{i-1})$, for all $i = 1, \ldots, k$. The assumption
that $W(f,g) = {\rm id}$ then implies that $z_k = z$ provided that $z$ is sufficiently close to $0 \in \C$.
Naturally the formal series at $0 \in \C$ associated to $W(f,h^{-1} \circ g \circ h)$ coincides with the
Taylor of the local diffeomorphism $\circ_{i=1}^k b_i$ at the origin, where $b_i$ is defined by
\begin{align*}
b_i &= a_i \qquad \qquad \qquad  \quad \,  \text{if} \, \, a_i = f^j \, \, \text{or} \, \, a_i = (f^{-1})^j \, \, \text{for some} \, \, j \in \N\\
b_i &= h^{-1} \circ a_i \circ h \qquad \text{otherwise} \, .
\end{align*}
Thus it suffices
to construct $h$ fulfilling the following condition: there exists a connected open neighborhood $U$ of
$0 \in \C$ where all maps $b_i$ are defined, as well as their compositions appearing
in the writing of $W(f,h^{-1} \circ g \circ h)$, and a point $z \in U$ such that $W(f,h^{-1} \circ g \circ h)(z)
\neq z$. To construct a local diffeomorphism $h$ satisfying this last condition, we proceed as follows.

Let us denote by $r$ the maximum of the distances from $\{z_i\}_{i=1, \ldots, k}$ to the origin. Modulo
performing a change of coordinates, we can assume without loss of generality that $r$ is attained by a
{\it unique}\, index $i_0$. Besides we have:
\begin{enumerate}
\item $0 < i_0 < k$
\item $a_{i_0 + 1} \in \{g^j, (g^{-1})^j, j \in \N)\}$
\end{enumerate}
As to condition~(2), observe that it does not affect the generality of the discussion since the group
generated by $f$ and $h^{-1} \circ g \circ h$ is conjugate to the group generated by $h \circ f \circ h^{-1}$
and $g$. In other words, we can permute the roles of $f$ and $g$ if necessary. Concerning condition~(1) we
should first note that $W(f,g) = a_k \circ \cdots \circ a_2 \circ a_1$ is the identity map if and only if
so is $\bar{W}(f,g) = a_1 \circ a_k \circ \cdots \circ a_2$. Now take $w = z_1$. The itinerary of $w$ with
respect to the new word $\bar{W}(f,g)$ coincides with the itinerary of $z$ with respect to $W(f,g)$. Indeed
the sequence of points involved in these itineraries are such that $w_i = z_{i+1}$ for all $0 \leq i < k$
and $w_k = z_1$. Now the set of points belonging to the itinerary of $w$ is such that the maximum of their
distance to the origin occurs at $w_{i_0-1}$. Summarizing condition~(1) can be assumed without loss of generality.
Indeed the argument allows us to assume in the sequel that $i_0 = k-1$.

Let $0 < \rho < r$ be such that the disk of radius $\rho$, $B_{\rho}$, contains all $\{z_i\}_{i=1, \ldots, k}$
with exception of $z_{i_0} = z_{k-1}$. The idea is to construct a univalent function $h$ that is ``very close"
to the identity map in $B_{\rho}$ and sends $z_{k-1}$ to a point $w_{k-1}$ sufficiently far from $z_{k-1}$ so
that $a_{k}(z_{k-1})$ and $a_{k}(w_{k-1})$ are ``far" from each other, but not ``too far" in the sense that
$a_{k}(w_{k-1})$ must still belong to $B_{\rho}$. Since $h$, and consequently $h^{-1}$, are ``close to
the identity" in $B_{\rho}$ it follows that the sequence of points formed by the itinerary of $z_0$ w.r.t.
$W(f, h^{-1} \circ g \circ h)$ will no longer be closed. Besides the diffeomorphism (univalent function) $h$ will be so that all the relevant
maps will still be defined on $B_{\rho}$. The construction of $h$ will be carried out with
the help of a sequence of Riemann mappings for appropriate simply connected domains. These domains are going to
be chosen so that the convergence on $B_{\rho}$ to the identity of the corresponding Riemann mappings will follow from
the classical theorem due to Caratheodory.

Now let us start to turn the above ideas into accurate statements. Up to a rotation and a re-scaling of coordinates we can assume
the following: the point $z_{k-1}$, i.e. the point of the sequence $\{z_i\}_{1 \leq i \leq k}$ with greatest
absolute value, is equal to~$1$. Besides both $f, \, g$ are defined on a disk of radius strictly greater than~$1$.
Fix $\delta > 1$, $0 < \varepsilon < 1$ and let
$D = \{z \in \C : |z| < 1\} \cup \{z \in \R : 1 \leq z \leq \delta\}$. We shall first prove:

\begin{prop}\label{sequencefunctions}
There exists a sequence of (holomorphic) univalent functions $H_m$ defined on a small neighborhood of the closed unit disk
and satisfying the following:
\begin{enumerate}
\item $H_m$ converges uniformly to the identity map on the disk of radius $1 - \varepsilon$.

\item $H_m(1) = \delta$.
\end{enumerate}
\end{prop}

\begin{proof}
Let $\delta$ and $\varepsilon$ be constants as above and consider also the previously defined set $D$.
Let $\{\alpha_m\}_{m \in \N}$ be a decreasing sequence of positive reals converging to zero as $m$ goes to
infinity and denote by $D_m$ the neighborhood of $D$ constituted by those points of $\C$ whose distance
to $D$ is at most $\alpha_m$. Since $D_m$ is a simply connected domain there exists a unique univalent function $h_m$
mapping $D_m$ conformally onto the unit disc and satisfying $h_m(0) = 0$ and $h_m^{\prime}(0) = 1$.
We are going to prove that the sequence of functions $h_m$ converge uniformly to the identity map on the disc
of radius $1 - \varepsilon$ and that $h_m(\delta) = p_m$ where $p_m$ is a real number and $\{p_m\}$ converges to $1$ as $m$ goes to infinity.
This is going to be sufficient to derive the statement of the proposition. In fact, let $B_{1/p_m}$ be
the disk of radius $1/p_m (> 1)$ and let $\lambda_m$ denote the homothety of ratio $p_m$. The homothety $\lambda_m$
maps $B_{1/p_m}$ conformally onto the unit disk and satisfies $\lambda_m(1) = p_m$. Now it can easily be checked that
$H_m$ given by $H_m = h_m^{-1} \circ \lambda_m$ satisfies all the required properties.

Thus the proof of our proposition is reduced to checking that the Riemann maps $h_m$ fulfils the mentioned
conditions. Namely $h_m(\delta) = p_m$ where $p_m \rightarrow 1$ as $m \rightarrow \infty$ and $h_m$ converges
uniformly to the identity on the disc of radius $1 - \varepsilon$. For this we shall be led to consider Riemann's
original argument which is slightly more quantitative than the standard modern proof. This quantitative aspect
will be important to show that $h_m(\delta)$ converges to~$1$ (whereas this convergence is highly intuitive).
Thus, following Riemann, denote $z$ by $x + iy$, where $x, y \in \R$, and consider the boundary
value problem
\[
\begin{cases}
\frac{\partial ^2 u}{\partial x^2} + \frac{\partial^2 u}{\partial y^2} = 0\\
u(x,y) = {\rm log} |x + iy | \, \, \, , \, \, \, (x,y) \in \partial D_n
\end{cases}
\]
The elementary theory of the Dirichlet problem ensures the existence of a unique solution for this boundary
value problem that, in addition, is bounded on $D_m$ (see for example \cite{who}). On the other hand, since
$u$ satisfies the Laplace equation, it follows the existence of a unique function $v$ defined on $D_m$ and
satisfying
\begin{equation}\label{CauchyRiemann}
\frac{\partial v}{\partial x} = -\frac{\partial u}{\partial y} \, \, \, \, \,  , \, \,  \, \, \,
\frac{\partial v}{\partial y} = \frac{\partial u}{\partial x} \, \, \, \, \, , \, \, \, \, \,
v(0,0) = 0 \, .
\end{equation}
Next set
\[
h_m(x+iy) = (x+iy) e^{-u(x,y) - iv(x,y)} \, ,
\]
and note that it satisfies the Cauchy-Riemann equations (Equation~\ref{CauchyRiemann}) so that $h_m$ is a
holomorphic function on $D_m$. Moreover $h_m$ can be written in the form
\[
h_m(x+iy) = e^{-p(x,y) - i q(x,y)} \, ,
\]
where $p(x,y) = u(x,y) - {\rm log} |x + iy|$ and $q(x,y)  = v(x,y) - {\rm arg} (x + iy )$ are defined
away from the origin. Riemann then proves that $h_m$ maps conformally $D_m$ into the unit disc. It follows,
in particular, that the level curves of $p$ are taken one-to-one into the circles of center at the origin.
Moreover the level of $p$ increases as the radius of its image decreases. The boundary of $D_m$ is in
particular sent to the boundary of the unit disk.

To prove that the sequence of Riemann mappings $h_m$ satisfies the conditions mentioned above, we first note
that $D_m$ is symmetric with respect to the real axis. Besides the boundary condition is also symmetric with
respect to the real axis. It then follows that the solution $u$ to the Dirichlet problem in question is
symmetric to the real axis as well. As a consequence $h_m$ leaves this axis invariant (for every $m \in \N$).
As already mentioned the levels of $p$ increases whereas the radius of its image decreases. We note however
that the levels of $p$ grows as we approach the origin $(0,0) \in \R^2$. This means that the segment contained
in the real axis joining $0$ to $\delta + \alpha_m$ is taken to the unit interval $[0,1]$ (contained in the real
axis) by an one-to-one monotone increasing real function. To conclude that $h_m (\delta)$ converges to~$1$
is now sufficient to show the maps $h_m$ converges uniformly to the identity on every compact subset of the
unit disk.

The uniform convergence will be derived from the statement of the classical Caratheodory convergence theorem
concerning convergence
of univalent functions, cf. for example \cite{duren}. Consider the above sequence of simply connected domains
$\{D_m\}_{m \in \N}$. Each $D_m$ is a proper domain of $\C$ containing the origin and $(h_m)^{-1}$ maps
conformally the unit disk into $D_m$. Moreover $(h_m)^{-1}(0) = 0$ and $h_m^{\prime}(0) = 1$. It is immediate
to check that $\mathbb{D} = \{z \in \Z : |z| < 1\}$ corresponds to the Kernel of $\{D_m\}_{m \in \N}$. Furthermore
every subsequence of $\{D_m\}_{m \in \N}$ has the same Kernel which means that $\{D_m\}$ converges to $\mathbb{D}$.
In particular $(h_m)^{-1}$ converges uniformly to a holomorphic function $g$ on each compact subset of $\mathbb{D}$.
Now note that $g$ maps $\mathbb{D}$ conformally onto $\mathbb{D}$, i.e. $g$ is an automorphism of the unit disk.
Moreover $g(0) = 0$ and $g^{\prime}(0) = 1$ since so does $(h_m)^{-1}$ for all $m \in \N$. This ensures that $g$ is
the identity map. It is also part of the statement of Caratheodory's theorem that $h_m$ converges uniformly on each
compact subset of $\mathbb{D}$ to the inverse of $g$, $g^{-1}$, i.e. it also converges to the identity map. In
particular $h_m$ is close to identity on the disk of radius $1 - \varepsilon$ for all $m \geq N$. This completes the
proof.
\end{proof}

\begin{obs}
\label{tangenttoidentity}
{\rm The contents of this remark will only be needed in Section~5 in connection with the use of Krull topology in the statement
of Theorem~C. Fixed $N \in \N^{\ast}$, we want to point out that the sequence $H_m$ constructed in the preceding proposition
can be supposed, in addition, to satisfy the following condition:

\noindent \hspace{1.0cm} $\bullet$ Each $H_m$ is tangent to the identity to order $N$ at $0 \in \C$.

\noindent To justify our claim, let $H_m$ be the sequence given by the proposition in question. In particular $H_m (0) =0$ and
$H_m^{\prime} (0) =1$. Since $N$ is fixed, we set
$$
H_m (z) = z + c_2^m z^2 + c_3^m z^3 + \cdots + c_N^m z^N + \cdots \, .
$$
For every $j \in \{2,\ldots ,N\}$ the sequence of coefficients $\{ c_j^m \}_{m \in \N^{\ast}}$ converges to zero as an immediate
consequence of Cauchy formula since $H_m$ converges to the identity on the disc of radius $1-\epsilon < 1$. Let then $R_m$
be the polynomial $R_m (z) = z+ c_2^m z^2 + c_3^m z^3 + \cdots + c_N^m z^N$ and denote by $S_m$ its inverse map. It is clear that
both $R_m , \, S_m$ converge to the identity on (fixed) arbitrarily large discs about $0 \in \C$. Hence the new sequence of maps
$\tilde{H}_m$ defined by
$$
\tilde{H}_m = S_m \circ H_m
$$
is tangent to the identity to order $N$ at $0 \in \C$. Besides this sequence clearly converges to the identity on the disc of radius
$1-\epsilon$. Finally, though we cannot ensure that $\tilde{H}_m (1) = \delta$, we know that $\tilde{H}_m (1) \rightarrow \delta$ and this
will suffice for our purposes.}
\end{obs}

We are now ready to prove the main result of this section. Let us however mention that the proof given below makes use of the
possibility of choosing a point $z$ arbitrarily close to the origin. This will be used to guarantee that the domains of definitions of certain
maps are connected and do contain the origin. A more general argument will be supplied in the next section.

\begin{proof}[Proof of Theorem~\ref{existence}]
Fix a pair of holomorphic diffeomorphisms $f, g \in \diff$. We shall prove that for every non-trivial
reduced word $W (a,)$ (different from $a^k$ and from $b^l$, $k,l \in \Z$),
there exists a formal series $\hat{h}_W \in \fdiff$ such that the element of $\fdiff$ corresponding to
$W(f, \hat{h}_W^{-1}\circ g \circ \hat{h}_W)$ does not coincide with the identity. So, from now on,
let as assume that a word as above $W$ is fixed and that $W(f,g) ={\rm id}$. Indeed $\hat{h}_W$
will be realized as the Taylor series of a convenient local diffeomorphism $h_W$.

To prove the existence of the diffeomorphism  $h_W$ let us follow the scheme described above. Fix a point
$z$ in a neighborhood of the origin and let $\{z_i\}_{i=0, \ldots, k}$ be the sequence of points formed by
its itinerary ($z = (z_0)$). Let $i_0$ denote the index for which $\{|z_i|\}_{i=0, \ldots, k}$ attains its strict
maximum, i.e. $\vert z_i \vert < \vert z_{i_0} \vert$ provided that $i \neq i_0$.
As already mentioned, there is no loss of generality in assuming that $i_0 = k - 1$ and
that a ``power" of $g$ or of $g^{-1}$ will be applied next. Moreover, up to rotating and re-scaling coordinates
we can also suppose that $z_{k-1}$ is
equal to~$1$ and that both $f, \, g$ are defined on a disk of radius greater than~$1$. Finally let $0 < \rho
< 1$, close to~$1$, be such that all $\{z_i\}_{i=0, \ldots, k}$ but $z_{k-1}$ are contained in $D_{\rho}$.

Let $U_{k-1}$ be a small disc centered at $z_{k-1}$ satisfying the following conditions:
\begin{itemize}
\item $a_k (U_{k-1}) \subseteq D_{\rho}$

\item $U_j = \circ_{i = k-1}^j a_i^{-1} (U_{k-1})$ is contained in $D_{\rho}$ for all $1 \leq j \leq k-1$.

\item The sets $\{U_j\}_{j = 1, \ldots, k-1}$ are pairwise disjoint.
\end{itemize}
By assumption $|z_i| < \rho$ for every $i \ne k-1$ implying that $U_{k-1}$ can be chosen as a non-empty
open set. Next fix $\delta \in \R$ ($\delta > 1$) as ``lying almost in the boundary of $U_{k-1}$" in the
sense that $(\delta - 1)/r$ is close to~$1$, where $r$ stands for the radius of $U_{k-1}$. Next consider
the elements defined in the following way: $\delta_{k-1} = \delta$ and $\delta_k = a_k (\delta_{k-1})$.
Naturally we have that $\delta_k$ belongs to $U_k$ but it is ``far" from $z_k$ in the sense that its
distance is bounded from below by a positive constant.

Let $\delta$ be as above and choose $\varepsilon$ to be equal to $(1-\rho)/2$. Now Proposition~\ref{sequencefunctions}
guarantees the existence of a sequence of univalent functions $H_m$, whose domains of definitions contain the unit
disc, which converges uniformly to the identity map on $B_{1-\varepsilon}$ and satisfies $H_m(1) = \delta$ for every $m$.

For $m$ sufficiently large, let us consider the elements in $\diff$ given by
\begin{align*}
b_i &= a_i \qquad \qquad \qquad  \quad \,  \text{if} \, \, a_i = f^j \, \, \text{or} \, \, a_i = (f^{-1})^j \, \, \text{for some} \, \, j \in \N\\
b_i &= H_m^{-1} \circ a_i \circ H_m \qquad \text{otherwise} \, .
\end{align*}
Note that although the domain of definition of $H_m$ is the disc of radius $1/p_m$ the same
does not apply to the map $b_i$, in the case that $a_i = g^j$ or $a_i = (g^{-1})^j$ for some $j \in \N^{\ast}$.
In fact $H_m$ maps $B_{1/p_m}$ conformally onto $D_m$ (the neighborhood of $D$ whose distance to $D$ is
at most $\alpha_m$) which means that the domain of definition of $H_m^{-1}$ is $D_m$. However the image of
$D_m$ by $a_i$ is not necessarily contained in $D_m$. Let $W_{i,m} = \{z \in D_m \, ; \; a_i(z) \in D_m\}$. The
domain of definition of $b_i$ is therefore $U_{i,m} = H_m^{-1}(W_{i,m})$.

Whereas the open set $U_{i,m}$ contains the origin, it is important to realize that it may be disconnected, i.e.
$a_i(D_m) \cap D_m$ may admit two or more connected components. However, in our case, this possibility
can be ruled out. Indeed, modulo re-scaling coordinates again (i.e. modulo working sufficiently near the origin), we may assume
that $g^j$ is defined on a neighborhood of the unit disc $D$. Besides, $g$ being a local diffeomorphism of $(\C, 0)$, a sufficiently small
disc $V$ about the origin is automatically taken by $g^j$ (or by $g^{-j}$) to a
star-shaped set $g^j(V)$. From this it follows that $W_{i,m}$ must be connected
so that the construction of $H_m$ detailed in Proposition~\ref{sequencefunctions} implies that $U_{i,m}$ is connected as well.
Next set
$U_m = \cap_{i \in I} U_{i,m}$ where $I$ denote the set of $1 \leq i \leq k$ as above, i.e. of $i$ such
that $a_i = g^j$ or $a_i = (g^{-1})^j$ for some $j$. The open set $U_m$ is also a connected neighborhood
of the origin. We define $h_m$ as being the restriction of $H_m$ to $U_m$.

Since $H_m$ converges to the identity map as $m$ increases, it is easy to check that
$w_k$, defined by $w_k = b_k (z_{k-1})$, becomes
arbitrarily close to $\delta_k$. In particular there exists $m \in \N$ so that $w_k$ is distinct
from $z_k$.

To finish the proof we just need to make some further comments. So far only the last point in
the itinerary of the initial point w.r.t. the word $W(f, H_m^{-1} \circ g \circ H_m)$ has been taken
into consideration.
In fact, we only kept track of the part of this itinerary coming after the point
of index $i_0$, i.e. of index $k-1$. However the set $\{a_i\}_{1 \leq i \leq k-2}$ may already contain
elements of the form $g^j$ or $(g^{-1})^j$ for some $j \in \N$. This means that taking $w_0 = z = z_0$,
the element $w_{k-1}$ of the sequence defined by $w_{i+1} = b_{i+1}(w_i)$ does not necessarily coincides
(and generically does not coincide) with the complex number~$1$. The element $w_{k-1}$ is not far from~$1$
but we have no total control on its image by $H_m$ or better by $h_m$. The easiest way to deal with this
difficulty is by ``slightly perturbing the departing point" as follows: assuming that $w_{k-1} = 1$ we go
backwards by following $b_i^{-1}$ to find out which point must be $w_0$. Observe that $a_{k-1}$ is given
by a ``power" of $f$ or of $f^{-1}$, by construction. This guarantees that $w_{k-2}$ belongs to $B_{\rho}$.
Furthermore it also shows that $w_{k-2}$ coincides with $z_{k-2}$. Since $h_m$ is very close to the identity
map on $B_{1-\varepsilon}$, and therefore in $B_{\rho}$, it follows that $w_i$ is ``very close" to $z_i$ for
every $i \in \{0, \ldots , k-3\}$
In fact, they can be made arbitrarily close modulo taking $m$ sufficiently large. We then consider $z = w_0$
instead of $z_0$ as departing point and the theorem follows. It only remains to ensure that all the $w_i$,
$0 \leq i \leq k-2$, constructed as above belong to the domain of definition of $b_{i+1}$. This is equivalent
to ensuring that all the $z_i$ belong to the same domain since $z_i$ and $w_i$ become close to each
other as $h_m$ converges to the identity map on $B_{\rho}$.

Fix $i \in \N$ and consider the element $z_i$ on the itinerary of $z_0$. It is sufficient to consider the
case where $a_i = g^j$ or $a_i = (g^{-1})^j$ for some $j$. Assume for a contradiction that $z_i$ does not
belong to the domain of definition of $b_{i+1}$.
This is equivalent to saying that $(g^j \circ H_m)(z_i)$ (or
$((g^{-1})^j \circ H_m)(z_i)$) does not belong to $D_m$. Hence neither does the point $z_{i+1} = g^j(z_i)$
since $H_m$ is close to the identity map on $B_{\rho}$. In particular $z_{i+1}$ does not belong to the
unit disc $D$. This contradicts the assumption that $z_{k-1} = 1$ is the point of the itinerary of $z_0$ having
greatest distance to the origin. The proof of the theorem is completed.
\end{proof}

As already pointed out, the preceding proof depends on the connectedness of the set $U_m$ what, in turn, can
easily be ensured modulo working ``very close to $0 \in \C$''. However the possibility of ``reducing the neighborhood of $0 \in \C$
will no longer be available in the proof of, for example, Theorem~A. Thus in the next section we shall provide a more general
argument that dispenses with the connectedness of sets playing a role analogous to the role of $U_m$ in the discussion above.

\section{Proofs for Theorems~A and~B}

Again let $W(f,g)$ be a reduced word on $f,g$ that is not of the forms $f^k$ or $g^l$ with
$k,l \in \Z$. The combination
of Theorem~\ref{AlgebraicTh} with Theorem~\ref{existence} implies the existence of an open dense set $\mathcal{H}_W\subset \diff$ such that,
whenever $h \in \mathcal{H}_W$, the composition $W(f,  h^{-1} \circ g \circ h)$ defines a local holomorphic diffeomorphism on a neighborhood
of $0 \in \C$ which does not coincide with the identity (on a neighborhood of the origin itself).

We say that $f$ (resp. $g$) is a diffeomorphism of period $k \in \N^{\ast}$ (resp. $l \in \N^{\ast}$) if $f^k ={\rm id}$ (resp. $g^l ={\rm id}$)
on a neighborhood of $0 \in \C$. Similarly $f$ is said to be of infinite order if $f^k$ does not coincide with the identity
on a neighborhood of $0\in \C$ for every $k \in \N^{\ast}$. In other words, the germ induced by $f^k$ at $0 \in \C$ is not the identity.
Naturally a similar definition applies to $g$. With this terminology, the above mentioned results immediately yield the following theorems:

\begin{teo}
\label{almostfinal1}
Suppose that $f,g$ as above are local diffeomorphism of infinite order and consider $\diff$ as the group of germs of local diffeomorphisms.
Then there exists a $G\delta$-dense set $\mathcal{H} \subset \diff$ such that,
whenever $h \in \mathcal{H}$, the group generated by the elements $f, h^{-1} \circ g \circ h$ is free in $\diff$. In other words, every reduced non-trivial word
$W(f,g)$ defines a local diffeomorphism that does not coincide with the identity on a neighborhood of $0\in \C$.
\end{teo}

\begin{teo}
\label{almostfinal2}
Suppose that $f,g$ as above have periods respectively $k,l \in \N^{\ast}$. Then, for $h$ belonging to some
$G\delta$-dense set $\mathcal{H} \subset \diff$, the group generated by $f, h^{-1} \circ g \circ h$ in the group of germs
of diffeomorphisms $\diff$ is isomorphic to the group defined by the presentation
$$
\langle a,b \; \, ; \; \, a^k =b^l ={\rm id} \rangle \, .
$$
\end{teo}

The preceding theorems however fall short to imply the statements given in the Introduction.
To explain the difference between the corresponding statements,
and also to prepare the way for the proofs of the latter, let us briefly recall the notion of {\it pseudogroup}.
Consider local diffeomorphisms $f,f^{-1}, g, g^{-1}$ as above that
are defined on an open neighborhood $V$ of $0\in \C$. We want to consider the {\it pseudogroup $\Gamma = \Gamma (f,g,V)$
generated by $f,f^{-1}, g, g^{-1}$ on $V$}
(in the sequel we shall only say the pseudogroup generated by $f,g$ and their inverses or simply by $f,g$ when no confusion is
possible). Let us make the definition of $\Gamma$ precise. Consider a reduced word $W(f,g)$ (in $f,g$ and in their inverses)
having the form $F_{s} \circ \cdots \circ F_{1}$ where each $F_{i}$, $i \in \{1, \ldots ,s\}$, belongs to
the set $\{ f^{\pm 1},g^{\pm 1}\}$. The domain of definition of
$W(f,g)$ as an element of $\Gamma$  consists of those points $z \in V$ such that
for every $1 \leq l < s$, the point $F_{l} \circ \cdots \circ F_{1} (z)$
belongs to $V$. Since $0\in \C$ is fixed by $f,g$, it follows that every word has a non-empty domain of
definition as element of $\Gamma$. By construction the domain of definition, besides non-empty, is also an open set.
It may however be {\it disconnected}. Therefore the preceding theorems only have a bear
in the connected components of these domains that contain $0 \in \C$. More precisely, suppose for example that $f,g$
are as in Theorem~\ref{almostfinal1}. Then, modulo conjugating $g$ by a generic element $h$, we can assume that all non-trivial
reduced words $W(f,g)$ as above are different from the identity {\it on the connected
component containing $0\in \C$ of their domains of definitions}. Since the domain of definition of $W(f,g)$ may
have more than one connected component, the
preceding results do not exclude the possibility of having the element $W(f,g)$ of $\Gamma$ coinciding with the identity
on a non-empty open subset of $V$. In particular
if $\Gamma$ is the pseudogroup associated to a foliation, we cannot yet derive the conclusions
about the topology of the leaves stated in Theorem~C. In fact, to complete the proof of Theorems~A and~B we need a
non-trivial strengthening of the results obtained in Sections~2 and~3.

To state the desired strengthened version of the mentioned results require however a little technical detour. For this let $f,g$ and $V$ be as
above so that the pseudogroup generated by $f,g$ (and their inverses) acts on $V$. Consider also a non-trivial reduced word $W(f,g)$
and denote by ${\rm Dom}_W (V)$ its domain of definition. To be able to take advantage of Baire property, we would like to have a statement
of type ``for an open dense set of local diffeomorphisms $h \in \diff$, the element $W(f,h^{-1} \circ g \circ h)$ of the pseudogroup generated
by $f, h^{-1} \circ g \circ h$ does not coincide with the identity on any connected component of ${\rm Dom}_W (V)$''. This statement however
makes no sense since the domain of definition of a given local diffeomorphism $h$ may be smaller than $V$ so that the pseudogroup
generated by $f, h^{-1} \circ g \circ h$ will not be defined on the whole $V$. This is the main reason explaining why
a more careful formulation of our statements is needed.

Let us begin by fixing an open disc $D$ about $0 \in \C$ such that $f,g, f^{-1}, g^{-1}$ are defined and univalent
on a neighborhood of $D$. As usual $\overline{D}$ will stand for the closure of $D$ while $\partial D$ will denote the
boundary of $D$.

Next we choose and fix once and for all a sequence $\{p_n\}_{n \in \N} \subset \C$ dense in $\C$ and such that no point
$p_n$ lies in $\partial D$.
Fixed $n$ and given a word $W (f,g)$ in $f,g$ and their inverses, let $\mathcal{U}_{n,W}$ denote the set formed by those
local diffeomorphisms $h$ at $0 \in \C$ for which one of the two possibilities below is verified:
\begin{enumerate}
  \item $p_n$ does not belong to the domain of definition of $W(f,h^{-1} \circ g \circ h)$ viewed as an element of
  the pseudogroup generated by $f,g$ {\it on the closed disc $\overline{D}$}.

  \item $p_n$ belongs to the domain of definition of $W(f,h^{-1} \circ g \circ h)$ viewed as an element of
  the pseudogroup generated by $f,h^{-1} \circ g \circ h$ {\it on the open disc $D$}. Furthermore $W(f,h^{-1} \circ g \circ h)$
  is required {\it not to}\, coincide with the identity on a neighborhood of $p_n$.
\end{enumerate}

As always we suppose that $W(f,g)$ is not reduced to an integral power of $f$ or $g$ (or of their inverses).
With this standard assumption we shall prove the following:

\begin{prop}
\label{lastversion1.1}
The set $\mathcal{U}_{n,W} \subset \diff$ is open and dense for the analytic topology.
\end{prop}

Note that the set $\mathcal{U}_{n,W}$ is clearly open since $p_n \in \C \setminus \partial D$. Thus the non-trivial
part of the proof of Proposition~\ref{lastversion1.1} consists of showing that $\mathcal{U}_{n,W}$ is dense in the
analytic topology. Assuming for the time being that the statement of this proposition holds and recalling that
a countable union of countable sets is itself countable, we can define a
$G_{\delta}$-dense set $\mathcal{U} \subseteq \diff$ by setting
$$
\mathcal{U} = \bigcap_{W (f,g)} \bigcap_{n=1}^{\infty} \mathcal{U}_{n,W} \, .
$$
Now suppose we are given $h \in \mathcal{U}$
and consider a neighborhood $V \subset D$ of $0\in \C$ where $f, h^{-1} \circ g \circ h$ are defined. In a more
accurate way, let us require that $V \subset D$ is such that all the sets $f(V)$,
$h(V)$, $g\circ h(V)$ and $h^{-1} \circ g \circ h (V)$ are defined and contained in $D$. The next step is to consider the
pseudogroup generated {\it on  $V$}\, by $f, h^{-1} \circ g \circ h$. Given a word $W(f,h^{-1} \circ g \circ h)$
not reduced to either $f^k$ or to $h^{-1} \circ g^l \circ h$ ($k,l \in \Z$), let ${\rm Dom}_{W} (V)$ denote
its domain of definition when the word is regarded as an element of the mentioned pseudogroup.
Now note that ${\rm Dom}_{W} (V)$ is clearly an open set so that all its connected components are open as well. Let $V_1$
be one of these connected components. Since $\{p_n\}$ is dense in $\C$, there exists $n_1$ such that $p_{n_1} \in V_1$.
Since $V \subseteq D$, $p_{n_1}$ belongs to the domain of definition of $W(f,h^{-1} \circ g \circ h)$
viewed, this time, as belonging to the pseudogroup generated {\it on $D$}\, by $f, h^{-1} \circ g \circ h$.
Because $h$ belongs to $\mathcal{U}$ and, in particular, to $\mathcal{U}_{n_1,W}$, it follows that
$W(f,h^{-1} \circ g \circ h)$ does not coincide with the identity on a neighborhood of $p_{n_1}$. Therefore
$W(f,h^{-1} \circ g \circ h)$ does not coincide with the identity on $V_1$. Since $V_1$ is an arbitrary connected
component of ${\rm Dom}_{W} (V)$, we have proved the following:

\begin{teo}
\label{lastversion2.2}
There exists a $G_{\delta}$-dense set $\mathcal{U}$ of $\diff$ and a neighborhood $V \subset \C$ of $0 \in \C$
such that, whenever $h \in \mathcal{U}$, the following holds: every word $W(a,b)$ different from $a^k$ or $b^l$ for which the element $W(f,h^{-1} \circ g \circ h)$
of the pseudogroup generated by $f, h^{-1} \circ g \circ h$ on $V$ coincides with the identity on some connected
component of its domain of definition must be the identity in the free group generated by $a,b$ (and their inverses).
\end{teo}

Now we can go back to the proof of Proposition~\ref{lastversion1.1}. Hence $p_n$ and $W(f,g)$ are fixed.
Let $h \in \diff$ be given. We need to prove that every neighborhood of $h$ in the analytic topology contains some
element $\tilde{h}$ belonging to $\mathcal{U}_{n,W}$. Naturally we assume that $h \not\in \mathcal{U}_{n,W}$ for
otherwise there is nothing to be proved. Since $h \not\in \mathcal{U}_{n,W}$, it follows from the construction of
$\mathcal{U}_{n,W}$ that $p_n$ belongs to the domain of definition of
$W(f,h^{-1} \circ g \circ h)$ viewed as an element of the pseudogroup generated on the closed disc $\overline{D}$
by $f,h^{-1} \circ g \circ h$. Furthermore we have the alternative:
\begin{itemize}
  \item There is at least one point in the itinerary of $p_n$ by $W(f,h^{-1} \circ g \circ h)$ that lies in the
  boundary $\partial D$ of $D$.
  \item $p_n$ belongs to the domain of definition of $W(f,h^{-1} \circ g \circ h)$ viewed as an element of
  the pseudogroup generated on the open disc $D$ and, in addition, $W(f,h^{-1} \circ g \circ h)$ coincides with
  the identity on a neighborhood of $p_n$.
\end{itemize}

When the first alternative holds, it is clear that the local diffeomorphism $h$ can be approximated in the analytic topology
by a sequence $\{h_i\}$ of local diffeomorphisms such that the itinerary of $p_n$ by $W(f,h_i^{-1} \circ g \circ h_i)$
is not contained in $\overline{D}$. In other words, $p_n$ does not lie in the domain of definition of
$W(f,h_i^{-1} \circ g \circ h_i)$ considered as an element of the pseudogroup generated by $f,h_i^{-1} \circ g \circ h_i$
on $\overline{D}$. Therefore $h_i \in \mathcal{U}_{n,W}$ for every $i$ and hence $h$ is a limit of elements in
$\mathcal{U}_{n,W}$.

Summarizing what precedes, to prove the denseness of $\mathcal{U}_{n,W}$ it suffices to check that a local diffeomorphism
$h$ as in the second possibility above can also be approximated by elements of $\mathcal{U}_{n,W}$. Thus from now on let
us suppose that $h$ satisfies the corresponding conditions. Namely the itinerary of $p_n$ by $W(f,h^{-1} \circ g \circ h)$
is contained in $D$ and $W(f,h^{-1} \circ g \circ h)$ coincides with the identity on a neighborhood of $p_n$.

Let $U \subset D$ be a domain with $C^1$-boundary (or real analytic boundary) $\partial U$ which still contains the itinerary of $p_n$ by
$W(f,h^{-1} \circ g \circ h)$. By resorting to standard transverse intersection constructions, the domain $U$
can be selected so that $\partial U$ intersects transversally the boundary of each connected component $U_j$, $j=1, 2, \ldots$,
of the domain of definition ${\rm Dom}_{W} (U)$ of $W(f,h^{-1} \circ g \circ h)$ now regarded as an element of the
pseudogroup generated on $U$ by $f,h^{-1} \circ g \circ h$. To check this claim note that the boundaries $\partial
U_j$ of the components $U_j$ are essentially determined by the preimage of $\partial U$ under sub-words of
$W(f,h^{-1} \circ g \circ h)$. In fact, if $z \in \partial U_j$ then there is a point $y$ in the itinerary of $z$
by $W(f,h^{-1} \circ g \circ h)$ lying in the boundary of $U$ or, equivalently, there is a sub-word
$\textsc{w} (f,h^{-1} \circ g \circ h)$ of $W(f,h^{-1} \circ g \circ h)$ (as in Section~3) such that
$\textsc{w} (f,h^{-1} \circ g \circ h) (z) =y \in \partial U$. Next since $U \subset D$, it follows that
$\textsc{w}(f,h^{-1} \circ g \circ h)$ is defined on a neighborhood of $z$.
Thus, about $z$, the neighborhood of $\partial U_j$ is identified to the preimage under $\textsc{w}(f,h^{-1} \circ g \circ h)$
of $\partial U$ on a neighborhood of $y$ (or rather to a finite union of these preimages).
It is now easy to carry over standard constructions to ensure that the above indicated transverse intersection takes place.
Next standard results on ``semi-analytic sets'' apply to ensure that the number of connected components $U_j$ is, indeed, finite.
The set of these connected components will then be denoted by $U_1, \ldots , U_r$.

At this point it is useful to state some simple ``stability'' properties of the previous construction with respect
to deformations of the local diffeomorphism $h$. These go as follows.

\begin{lema}
\label{continuousmoving}
If $\tilde{h}$ is another element of $\diff$ sufficiently close to $h$ in the analytic topology
then (the maps $f$ and $\tilde{h}^{-1} \circ g \circ \tilde{h}$ are defined on some neighborhood of $U$ and)
the following holds:
\begin{enumerate}
\item The domain of definition of $W(f,\tilde{h}^{-1} \circ g \circ \tilde{h})$ as element of the pseudogroup generated on $U$ by
$f, \, \tilde{h}^{-1} \circ g \circ \tilde{h}$ contains exactly $r$ connected components denoted by
$\tilde{U}_1, \ldots , \tilde{U}_r$. Besides the boundaries of these components still transversally intersects $\partial U$.

\item Each of the connected components $\tilde{U}_j$ converges (in Hausdorff topology) towards $U_j$ as $\tilde{h}$
converges to $h$.
\end{enumerate}
\end{lema}

\begin{proof}
The proof amounts to standard continuity arguments. Details are left to the reader.
\end{proof}

Recalling that the itinerary of $p_n$ under $W(f,h^{-1} \circ g \circ h)$ is contained in $U$, we
can suppose, modulo renumbering the components $U_j$, that $p_n \in U_1$.

Note that, by assumption, the Taylor series of $W (f, h^{-1} \circ g \circ h)$ centered at $p_n$ is reduced to the identity.
On the other hand, the itinerary of $p_n$ under $W (f, h^{-1} \circ g \circ h)$ is entirely contained in $U \subset D$ by construction.
In particular $U$ is contained in a disc where the Taylor series of both $f, h^{-1} \circ g \circ h$ converge.
Note however that the coefficients of the Taylor series of $W (f, h^{-1} \circ g \circ h)$ at $p_n$ are obtained as functions
in {\it all the coefficients}\, of the Taylor series of $f,h^{-1} \circ g \circ h$ at $0 \in \C$.
For example $(W (f,h^{-1} \circ g \circ h))' (p_n)$ is not an algebraic function of finitely many coefficients of
$f,h^{-1} \circ g \circ h$ at $0 \in \C$ as in the Section~3. In fact, the computation of the coefficient in question involves all
(or infinitely many) Taylor coefficients of $f,h^{-1} \circ g \circ h$ at $0 \in \C$. To remedy
for this new difficulty, we shall modify the structure of the proof given in the previous section.

\begin{proof}[Proof of Proposition~\ref{lastversion1.1}]
As mentioned we need to find, arbitrarily near $h$, a local diffeomorphism $\tilde{h}$ such that
$W(f, \tilde{h}^{-1} \circ g \circ \tilde{h})$ will no longer coincide with the identity on a neighborhood
of $p_n$. Actually note that $W(f, \tilde{h}^{-1} \circ g \circ \tilde{h})$ is clearly defined about $p_n$ modulo choosing
$\tilde{h}$ close enough to $h$. The desired element
$\tilde{h}$ will be constructed in three different steps.

\noindent {\bf Step 1}. Construction of a first perturbation.

\noindent First we are going to construct a local diffeomorphism $\overline{h}$ as in the proof of Theorem~\ref{existence} (see also
Proposition~\ref{sequencefunctions}). By construction of $\overline{h}$ (noted $h_W$ in the proof of Theorem~\ref{existence}), it is clear that the
itinerary of $p_n$ under $W (f, \overline{h}^{-1} \circ g \circ \overline{h})$ is well-defined.
In other words, $p_n$ belongs to the domain of definition
of $W (f, \overline{h}^{-1} \circ g \circ \overline{h})$ considered as element of the pseudogroup generated by
$f,\overline{h}^{-1} \circ g \circ \overline{h}$ on a suitable domain contained in $V \subset D \subset \C$. Note however that this
time we have no information concerning the ``size'' of the connected component of the domain of definition of
$W (f, \overline{h}^{-1} \circ g \circ \overline{h})$ that contains $p_n$. Similarly we cannot resort to arguments
involving convexity as done in Section~3.

In any case, we keep the information that the itinerary of $p_n$ is well-defined and, besides, we have
$W (f, \overline{h}^{-1} \circ g \circ \overline{h}) (p_n) \neq p_n$. Next let us resume
the setting and the notations employed in the proof of Theorem~\ref{existence}. Thus ``$p_n$'' is identified with the point
``$z$'' of the mentioned proof. The itinerary of $p_n$ through $W(f, h^{-1} \circ g \circ h)$
is going to be denoted by $\{p_n^i\}_{i=1, \ldots, l}$ where
$p_n =p_n^1$. Without loss of generality we can suppose that $p_n^{l-1}$ is
the unique point of this itinerary having a maximal absolute value.
Consider now a family of maps $H_m$ similar to the corresponding family considered in Theorem~\ref{existence}.
In particular we have:
\begin{enumerate}
\item Each $H_m$ is defined on an open disc of radius slightly larger than $\Vert p_n^{l-1} \Vert$ and the resulting sequence
converges uniformly to the identity
on compact parts of $B (\Vert p_n^{l-1} \Vert)$ (the open disc of radius $\Vert p_n^{l-1} \Vert$ about the origin).

\item $\Vert H_m (p_n^{l-1}) - p_n^{l-1} \Vert = \tau > 0$ where $\tau$ is such that $H_m (p_n^{l-1})$ has properties similar
to the properties of ``$\delta >1$'' in the mentioned proof.

\item The point $p_n^{l}$ of the itinerary of $p_n$ is given by $p_n^{l} = h^{-1} \circ g^j \circ h (p_n^{l-1})$.
Besides $h^{-1} \circ g^j \circ h (H_m (p_n^{l-1}))$
lies in a compact part of the disc $B (\Vert p_n^{l-1} \Vert)$. In particular, for $m$ very large, $H_m^{-1}$ is defined
and close to the identity around $h^{-1} \circ g^j \circ h (H_m (p_n^{l-1}))$.
\end{enumerate}

As already observed, for $m$ very large $W (f, H_m^{-1} \circ h^{-1} \circ g \circ h \circ H_m) (p_n) \neq p_n$.
Set $\overline{h} = H_m$ for very large $m$ to be fixed later on.

In view of Remark~\ref{tangenttoidentity}, the sequence $H_m$ can be constructed so that, in addition
to the preceding conditions, each $H_m$ is tangent to the identity at $0 \in \C$ to an arbitrary order fixed
from the beginning. This information will be used to provide a variant of Theorem~\ref{lastversion2.2} well adapted
for the discussion carried out in Section~5.

\noindent {\bf Step 2}. Perturbing the Taylor series of $H_m$.

\noindent To guarantee convergence in the analytic topology it is necessary to be able to perform a number
of operations with the Taylor series of $H_m$. For example, if $h$ were the identity, then it would be important
to replace $H_m$ by suitable polynomials somehow as already done in Section~2. Here a slightly more elaborated
construction is going to be needed. To begin with, we would like to replace $H_m$ by
some polynomial $R$ with $R(0) =0$. In other words, we want to find a polynomial
$R$ {\it arbitrarily close to $H_m$ on discs of radius less than $\Vert p_n^{l-1} \Vert$}\, for which
$W (f, R^{-1} \circ h^{-1} \circ g^j \circ h \circ R)$ is defined about $p_n$
and, in fact, verifies $W (f, R^{-1} \circ h^{-1} \circ g \circ h \circ R) (p_n) \neq p_n$. The main reason why the existence of $R$
cannot be ensured by merely truncating the Taylor
series of $H_m$ at a sufficiently high order lies in the fact that the point $H_m (p_n^{l-1})$ does not belong
to the convergence disc of the Taylor series of $H_m$ at the origin. Indeed this radius of convergence is ``essentially'' given
by $\Vert p_n^{l-1} \Vert$ so that it is already unclear whether or not $p_n^{l-1}$ belongs to it.

To overcome the above difficulty we proceed as follows. Denote by $U^{H_m}$ the connected component of the
domain of definition of $W (f, H_m^{-1} \circ h^{-1} \circ g \circ h \circ H_m)$, w.r.t. $U$, containing $p_n$. Note that
the construction of $H_m$ makes $p_n$ to be close to the boundary of $U^{H_m}$ when $m$ is large. Also
recall that $U_1$ stands for the component of the domain of $W(f,h^{-1} \circ g \circ h)$, w.r.t. the set $U$ that contains $p_n$.
We are going to show the existence of another point $q_n \in U^{H_m}$ whose corresponding iterations by the elements
$f,h^{-1} \circ g \circ h,\, H_m$ and $H_m^{-1}$ as they
appear listed in $W (f, H_m^{-1} \circ h^{-1} \circ g \circ h \circ H_m)$ are all contained in a compact part of
$B (\Vert p_n^{l-1} \Vert)$.
Let then $l$ denote the segment of straight line delimited by $0 \in \C$ and by
$p_n^{l-1}$. Let $l'$ be the line issued from
$p_n$ corresponding to the pre-image of $l$ by the derivative of the diffeomorphism $\tilde{F}_m$ given as the sub-word of
$W(f, H_m^{-1} \circ h^{-1} \circ g \circ h \circ H_m)$ taking
$p_n$ to $p_n^{l-1}$. Note that all the maps $f,h^{-1} \circ g \circ h$ and $H_m, H_m^{-1}$ involved in the constitution
of $\tilde{F}_m$ have bounded $C^2$-norms on a neighborhood
of the points where they appear in the composition in question. In other words, the following holds: as we move from $p_n$ to $p_n^t$
over $l'$, the image $\tilde{F}_m (p_n^t)$ moves inward $B (\Vert p_n^{l-1} \Vert)$.
Besides this movement is uniform in the following sense: if $q_n$ is a given point ``close to $p_n$'' over $l'$, then
\begin{equation}
\Vert \tilde{F}_m (q_n) \Vert < \Vert p_n^{l-1} \Vert - \epsilon \label{isntoveryet}
\end{equation}
for some small $\epsilon >0$ where the constants in question do not depend on $m$. In fact, as $m$ increases, $H_m$ converges
to the identity around the relevant points appearing in the definition of $\tilde{F}_m (q_n)$.
In particular it follows that $q_n$ belongs to $U^{H_m}$. Finally since $H_m$ converges uniformly to the identity on
compact parts of $B (\Vert p_n^{l-1} \Vert)$, Estimate~(\ref{isntoveryet}) allows us to conclude that
$H_m (\tilde{F}_m (q_n))$ lies in $B (\Vert p_n^{l-1} \Vert)$ for $m$ large. The same then applies to $g^j (H_m (\tilde{F}_m (q_n)))$.

Summarizing what precedes, the following lemma was proved:

\begin{lema}
There is a point $q_n$ in $U_1 \cap U^{H_m}$ whose itinerary by $W (f, H_m^{-1} \circ h^{-1} \circ g \circ h \circ H_m)$
is entirely contained in a compact part of the convergence discs
of the Taylor series at $0 \in \C$ of $f,h^{-1} \circ g \circ h$ and $H_m, H_m^{-1}$.
In addition $W (f, H_m^{-1} \circ h^{-1} \circ g \circ h \circ H_m)(q_n) \neq q_n$.\qed
\end{lema}

Now modulo truncating the Taylor series of the elements $H_m$ at large orders $d(m)$, it follows the existence of
polynomials $R_m$ satisfying the conditions below.
\begin{enumerate}
\item The itineraries of $q_n$ under $W (f, H_m^{-1} \circ h^{-1} \circ g \circ h \circ H_m)$ and under
$W (f, R_m^{-1} \circ h^{-1} \circ g \circ h \circ R_m)$ are well-defined. Furthermore
$W (f, R_m^{-1} \circ h^{-1} \circ g \circ h \circ R_m) (q_n) \neq q_n$.

\item Both $p_n, q_n$ belong to the same connected component $U_1$ (resp. $U^{H_m}$) of the domain of definition of
$W(f,h^{-1} \circ g \circ h)$ (resp. $W (f, H_m^{-1} \circ h^{-1} \circ g \circ h \circ H_m)$).
\end{enumerate}

The contents of item~2
may not hold for $W (f, R_m^{-1} \circ h^{-1} \circ g \circ h \circ R_m)$, i.e. $p_n$ and $q_n$ may not belong to the same connected
component of the domain of definition of $W (f, R_m^{-1} \circ h^{-1} \circ g \circ h \circ R_m)$.
In fact, whereas
$R_m$ is close to $H_m$ on a suitable disc (not containing the itinerary of $p_n$), we have no control of the degree $d(m)$ of $R_m$. Hence
we cannot conclude that all its coefficients are simultaneously ``small''. In other words $R_m$ may become ``far'' from the
$H_m$ on a disc of radius slightly larger than the above mentioned disc. Since $H_m$ converges to the identity on both discs,
the preceding can be re-stated by saying that $R_m$ may be close to the identity on the smaller disc but far from the identity
on the larger disc. This explains why,
in principle, it is unclear even whether $W (f, R_m^{-1} \circ h^{-1} \circ g \circ h \circ R_m)$ is defined at $p_n$.

\noindent {\bf Step 3}. Adjusting perturbations.

\noindent The preceding construction has still one additional property. Namely the existence of a uniform disc ${\bf B}$
containing the itinerary of $q_n$ under $W (f, R_m^{-1} \circ h^{-1} \circ g \circ h \circ R_m)$. Here ${\bf B}$ was called
uniform in the sense that it does not depend on either $m$ or $d(m)$.
Besides the radius of ${\bf B}$ is smaller than the convergence radii of the Taylor series of
$f, \, h^{-1} \circ g \circ h,  H_m$. Finally
$H_m$ (resp $R_m$) converges uniformly to the identity on ${\bf B}$ when
$m \rightarrow \infty$. In particular ${\bf B}$ also
contains the itineraries of $q_n$ under $W (f, H_m^{-1} \circ h^{-1} \circ g \circ h \circ H_m)$
and under $W(f,h^{-1} \circ g \circ h)$.

Next let $h \circ R_m$ be truncated at large orders $N(m)$ so as to have the following condition
verified: the series formed by the components of degree higher than $N(m)$ of the Taylor series
of both $h \circ R_m$ and $h$ converge uniformly to {\it zero}\, on ${\bf B}$. Next we add to the truncation
of $h \circ R_m$ the terms having degree greater than $N (m)$ of the Taylor series of $h$. The result of
these operations is a new sequence of local diffeomorphisms $S_m$ with the properties indicated below:
\begin{description}

\item[(a)] $S_m$ (is defined and) converges uniformly to $h$ on ${\bf B}$ when $m \rightarrow \infty$. In particular the itinerary
of $q_n$ under $W (f, S_m^{-1} \circ g \circ S_m)$ is well-defined and contained in ${\bf B}$. Furthermore
$W (f, S_m^{-1} \circ g \circ S_m) (q_n) \neq q_n$.

\item[(b)] The Taylor series of $S_m$ and $h$ centered at $0 \in \C$ differ only by finitely many
coefficients.

\end{description}
In other words, the difference $S_m -h$ is represented on ${\bf B}$ by a polynomial denoted
by $Q_m$. Moreover $Q_m$ converges uniformly to {\it zero}\, since $S_m \rightarrow h$ uniformly
on ${\bf B}$.

Now let us fix $\tau >0$ small enough to ensure that the following condition
holds: whenever $\xi$ is a univalent
map on ${\bf B}$ whose distance to $h$ is less than $\tau$, the itinerary of $q_n$ under
$W (f, \xi^{-1} \circ g \circ \xi)$ is well-defined and contained in ${\bf B}$.
Finally we fix $m$ so large that $S_m$, in addition to verify all the previously determined conditions,
is also $\tau$-close to $h$ on ${\bf B}$. The degree $N(m)$ of the corresponding polynomial
$Q_m$ is going to be denoted simply by $N$. As already pointed out above, though $Q_m$ is close to {\it zero}\,
on ${\bf B}$, we cannot conclude that its coefficients are ``small'' since we have little control on
its degree~$N$.

Since $S_m (0) = h(0) =0$, we have $Q_m= c_1z + c_2 z^2 + \cdots + c_N z^N$. The quantity
$W (f, S_m^{-1}  \circ g \circ S_m) (q_n) -q_n$ can then be
regarded as an analytic function on the coefficients $c_1 ,\ldots ,c_N$ of $Q_n$ through the identification
$S_m = h + Q_m$. In principle this function is defined on a neighborhood of $(c_1, \ldots ,c_N)$. Setting however
$S_m^t(z) = h+  tc_1 z + \cdots + tc_N z^N = h+tQ_m$ it follows that the function is question is well-defined at
$(tc_1, \ldots , tc_N)$ for every $t \in [0,1]$ since the distance from $S_m^t$ to $h$ is less than $\tau$ for every
$t \in [0,1]$. Indeed this distance is nothing but $t \sup \Vert Q_m \Vert \leq  \sup \Vert Q_m \Vert \leq \tau$.
Thus the analytic function in question is defined on some small neighborhood of the
segment of line given by $t \mapsto (tc_1, \ldots , tc_N)$, $t \in [0,1]$. The resulting function of variable
``$t$'' however is not identically zero since for $t=1$ we have
$W (f, S_m^{-1} \circ g \circ S_m) (q_n) \neq q_n$, cf. item~(a) above.
Therefore we can find $N$-tuples $(c_1^{t}, \ldots ,c_N^{t})$, $c_i^t = tc_i$, arbitrarily close to $(0, \ldots ,0)$ such that
$W (f, (S_m^{t})^{-1} \circ g  \circ S_m^{t}) (q_n) \neq q_n$. To complete the proof it suffices to check that
$S_m^t$ fulfils all the conditions of the statement provided that $t$ is close enough to $0$.

For this note first that $S_m^t$ clearly converges to $h$ in the analytic topology when
$t \rightarrow 0$. In particular $S_m^t$ converges uniformly to $h$ on the whole $U$. Thus for $t$ small enough,
$S_m^t$ satisfies the conditions of Lemma~\ref{continuousmoving}. Hence the domain of definition of
$W (f, (S_m^{t})^{-1} \circ g  \circ S_m^{t})$ as element of the pseudogroup generated on $U$ by
$f, (S_m^{t})^{-1} \circ g  \circ S_m^{t}$ contains exactly $r$ connected components
$\tilde{U}_1^t, \ldots , \tilde{U}_r^t$ in natural correspondence with the components $U_1, \ldots ,U_r$ of the domain of definition of
$W (f, h^{-1} \circ g  \circ h)$ as element of the analogous pseudogroup. Besides the domain $\tilde{U}_1^t$ converges suitably
to $U_1$ when $t \rightarrow 0$.
Since $p_n, q_n$ belong to the same connected component of the domain of $W (f, h^{-1} \circ g  \circ h)$, cf. item~2 above,
we conclude that $q_n$ belongs to $U_1$. Therefore, modulo choosing $t$ very small, it follows that both
$p_n, \, q_n$ belong also to $\tilde{U}_1^t$. Therefore
the restriction of $W (f, (S_m^{t})^{-1} \circ g  \circ S_m^{t})$ to $\tilde{U}_1^t$ does not coincide with the identity
since $W (f, (S_m^{t})^{-1} \circ g  \circ S_m^{t}) (q_n) \neq q_n$. Because this is an analytic map defined on $\tilde{U}_1^t$
it follows that $W (f, (S_m^{t})^{-1} \circ g  \circ S_m^{t})$ cannot coincide with the identity on a neighborhood of $p_n$.
This completes the proof of Proposition~\ref{lastversion1.1}
\end{proof}

The preceding proof also yields the following variant of Theorem~\ref{lastversion2.2} which is necessary to be able to turn generic
statements about local diffeomorphisms equipped with the analytic topology into generic statements about foliations
with respect to the Krull topology. For this let $N \in \N^{\ast}$ be fixed and consider the subgroup
$\diff_N \subset \diff$ consisting of those elements that are tangent to the identity at $0 \in \C$ to an order
at least $N$. Cauchy Formula shows that $\diff_N$ is closed with respect to the analytic topology. It is also clear
that $\diff_N$ endowed with the restriction of the analytic topology still is a Baire space. Now we have:
Consider elements $f,g$ in $\diff$ of finite order respectively equal to $k,l \geq 1$. Through the substitutions
$a^{\pm 1} \mapsto f^{\pm 1}$ and $b^{\pm 1} \mapsto g^{\pm 1}$, words $W (a,b)$ representing elements of the group
$\Z /k \Z \ast \Z /l\Z$ can be identified to the holomorphic maps $W(f,g)$.

\begin{teo}
\label{lastversion3.3}
Consider elements $f,g$ in $\diff$ or finite orders respectively equal to $k,l \geq 1$. Then
there exists a $G_{\delta}$-dense set $\mathcal{U}_N$ of $\diff_N$ and a neighborhood $V \subset \C$ of $0 \in \C$ such that,
whenever $h \in \mathcal{U}_N$, the following holds: every word $W(a,b)$ for which the element
$W(f,h^{-1} \circ g \circ h)$ of the pseudogroup generated by $f, h^{-1} \circ g \circ h$ on $V$ coincides with the identity os some connected
component of its domain of definition must represent the identity $\Z /k \Z \ast \Z /l\Z$.
\end {teo}

\begin{proof}
Consider a non-trivial reduced word in $W (f,g)$ and let it be further reduced through the identities
$f^k =g^l ={\rm id}$. The new word is still denoted by $W (f,g)$ and we can suppose that it is not reduced to
a integral power of either $f$ or $g$. We just need to prove a version of Proposition~\ref{lastversion1.1}
relative to $\diff_N$. We repeat the same structure of proof, in particular we have ``relative'' sets
$\mathcal{U}_{n,W}$ which are still open. It is then enough to check that these sets are also dense in
$\diff_N$. For this we resume the argument given in the steps~1, 2 and~3 of the proof of Proposition~\ref{lastversion1.1}.
As to Step~1, Remark~\ref{tangenttoidentity} allows us to construct the sequence $H_m$ so that each
$H_m$ is tangent to the identity at $0 \in \C$ to order $N$. The truncations of Taylor series performed
in Step~2, in particular the one used to define the polynomials $R_m$ can equally well be performed in $\diff_N$.
Finally, concerning Step~3, the additional truncations of Taylors series can also be carried out in $\diff_N$,
in particular the local diffeomorphism $S_m$ and the polynomial $Q_m$ may be supposed to belong to $\diff_N$.
The rest of the proof carries over automatically to the present context.
\end{proof}

We are finally able to prove the first two theorems stated in the Introduction.

\begin{proof}[Proof of Theorems A and B]
The first two parts of these statements are contained in Theorems~\ref{almostfinal1} and~\ref{almostfinal2}
and in Theorem~\ref{lastversion2.2}.
It only remains to prove the last part of these statements. Namely we need to
show the existence of infinitely many points $\{ \textsc{z}_i \}$ in the neighborhood $V$ of $0 \in \C$ given by
Theorem~\ref{lastversion2.2} satisfying the following conditions:
\begin{enumerate}
\item Each $\textsc{z}_i$ is fixed by an element $W_i (f,g)$ (whose domain of definition contains $\textsc{z}_i$) which does not coincide
with the identity on any neighborhood of $\textsc{z}_i$.

\item If $i\neq j$ then $\textsc{z}_i$ and $\textsc{z}_j$ possess disjoint orbits under the pseudogroup generated by $f,g$
on $V$.

\end{enumerate}

The construction of these points goes through ideas already appearing
in \cite{Frank1}, \cite{Morefrank} (the reference \cite{julio} might also be useful). We sketch the argument in the sequel.
A generic choice of $h$ allows us to suppose that the pseudogroup generated by $f,g$ verifies the conclusion of
Theorem~\ref{lastversion2.2}. In particular its {\it germ}\, at $0 \in \C$ is a non-solvable group.
Secondly, by using again our perturbation techniques, the following claim can be proved (details are left to the reader):

\noindent {\it Claim 1}. For a generic choice of $h$ (and substitution of $g$ by $h^{-1} \circ g \circ h$) we can suppose
that every point $z \in V \setminus \{ 0\}$ has centralizer either trivial or infinite cyclic. In other words,
if $W_1 (f,g), \, W_2 (f,g)$ are such that $W_1 (f,g) (z) = W_2 (f,g) (z) =z$,  then there is another element $W_{1,2}$, defined
about $z$ and verifying such that $W_1 (f,g) = (W_{1,2} (f,g))^{k_1}$ and $W_2 (f,g) = (W_{1,2} (f,g))^{k_2}$
for suitable integers $k_1, k_2$.\qed

Similarly we can also suppose that the following ``non-commutative condition'' is satisfied:

\noindent {\it Claim 2}. Suppose that $W_1 (f,g), \, W_2 (f,g)$ are elements of the pseudogroup generated
on $V$ by $f,g$ that are defined about a point
$z \in V$. Suppose also that $(W_1 (f,g))^{-1} \circ W_2 (f,g) \circ W_1 (f,g)$ is defined about $z$ and, in a neighborhood
of $z$, satisfies $(W_1 (f,g))^{-1} \circ W_2 (f,g) \circ W_1 (f,g) = W_2 (f,g)$. Then we must have
$W_1 (f,g) = (W_{1,2} (f,g))^{k_1}$ and $W_2 (f,g) = (W_{1,2} (f,g))^{k_2}$ for some element $W_{1,2}$ and certain
integers $k_1, k_2$.\qed

After \cite{Frank1}, we are ensured of the existence of $\textsc{z}_1$ as above which, in fact, is a hyperbolic fixed point for some element
$W_1 (f,g)$. We shall construct points $\textsc{z}_1', \textsc{z}_2'$ ($\textsc{z}_1' \neq \textsc{z}_2'$) with similar properties and
such that $\textsc{z}_1', \textsc{z}_2'$ have disjoint orbits. The proof will then easily follow by an inductive argument. For this note that,
according to the mentioned works and thanks to
the contractive character of $W_1 (f,g) $ at $\textsc{z}_1$, there exists a neighborhood $U_1$ of $\textsc{z}_1$
such that any holomorphic map $F$ from $U_1$ to $F (U_1) \subseteq V$ can be approximated on compact sets of $U_1$ by actual elements
$\textsc{w}^1 (f,g), \textsc{w}^2 (f,g) ,\ldots$, defined on $U_1$ and belonging to the pseudogroup generated by $f,g$ on $V$.
In particular $F$ can be chosen so as to have two (hyperbolic) fixed points, namely $\textsc{z}_1$ and another point
$\textsc{z}_2$ both lying in $U_1$. Thus for $K$ sufficiently
large, the element $\textsc{w}^K (f,g)$ of $\Gamma$ has at least two distinct fixed points in $U_1$.
These points are denoted by $\textsc{z}_1', \textsc{z}_2'$. It remains only to show
that the orbits of $\textsc{z}_1', \textsc{z}_2'$ under the pseudogroup in question are disjoint.
Suppose for a contradiction this is not the case. Then there exists an element
$W (f,g)$ such that
\[
(W(f,g))^{-1} \circ \textsc{w}^K (f,g) \circ W(f,g) (\textsc{z}_1') =\textsc{z}_1' \, .
\]
In particular
$(W (f,g))^{-1} \circ \textsc{w}^K (f,g) \circ W(f,g)$ is defined on a neighborhood of $\textsc{z}_1'$.
In view of the conditions about stabilizers of points in $V$, it follows that
$(W (f,g))^{-1} \circ \textsc{w}^K (f,g) \circ W(f,g)$ must coincide with a power of $\textsc{w}^K (f,g)$.
Therefore $W (f,g)$ is itself a power of $\textsc{w}^K (f,g)$ (on a neighborhood of $\textsc{z}_1'$) by
virtue of Claim~2. This is however impossible since
$\textsc{w}^K (f,g) (\textsc{z}_1') =\textsc{z}_1'$ and $W (f,g) (\textsc{z}_1'') =\textsc{z}_2' \neq \textsc{z}_1'$.
The proofs of Theorems~A and~B are now complete.
\end{proof}

\section{An application to nilpotent foliations}

As indicated in the Introduction, the problem of perturbing the generators of a subgroup of $\diff$ inside their conjugacy classes
arises naturally in the study of germs of singular foliations at the origin of $\C^2$. Probably the most typical example
where this situation can be found corresponds to the class of nilpotent foliations of type
$A^{2k+1}$. More precisely, these are local foliations $\fol_{\om}$ defined by
a (germ of) $1$-form $\om$ having nilpotent linear part, i.e. $\om = y dy + \cdots$, and a unique separatrix $S$ that happens
to be a curve analytically equivalent to $\{ y^2 - x^{2k+1} = 0\}$. In other words, there are local coordinates where $S$ is given
by the equation $\{y^2 - x^{2k+1} = 0\}$. For this type of foliation the
desingularization of the separatrix ``coincides" with the reduction of the foliation itself. In other words,
the map associated to the desingularization of the separatrix $\Pi_S: M \rightarrow \C^2$ reduces also
the foliation $\fol_{\om}$. The exceptional divisor consists of a string of $k+2$ rational curves whose
dual graph is

\begin{figure}[hbtp]
\centering
\includegraphics[scale=0.8]{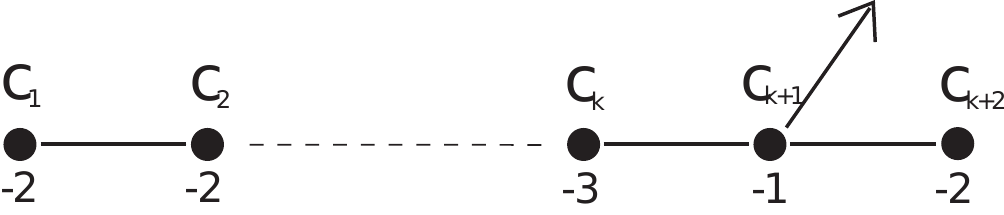}
\caption{The desingularization diagram of the foliation}
\label{graph}
\end{figure}

The vertices of this graph correspond to the irreducible components of the exceptional divisor $\mathcal{D}
= E_S^{-1}(0)$. The weight of each irreducible component corresponds to its self-intersection. In turn the edges
correspond to the intersection of two irreducible components and, finally, the arrow corresponds to the
intersection point of the (unique) component $C_{k+1}$ of self-intersection~$-1$ with the transform $\tilde{S}$ of $S$.
The component $C_{k+1}$ contains three singular points $s_0, s_1$ and $s_2$ where $s_0$ corresponds to the point determined
by the intersection of $C_{k+1}$ with $\tilde{S}$. Finally $s_1$ (resp. $s_2$) is the intersection point of $C_{k+1}$
with $C_{k+2}$ (resp. $C_k$).

Denote by $\tilf$ the transform of $\fol_{\om}$. Note that the holonomy associated to the component $C_{k+2}$,
i.e. the holonomy map associated to the regular leaf $C_{k+2} \setminus \{ s_1\}$ of $\tilf$, coincides with the identity
since this leaf is simply connected. It then follows that the germ of $\tilf$ at $s_1$ admits a holomorphic first integral.
Since the corresponding eigenvalues are $1, \, 2$ we conclude that the local holonomy map $g$ associated to a small loop around
$s_1$ {\it and contained in $C_{k+1}$}\, has order equal to~$2$. A similar discussion applies to the component $C_1$ and
leads to the conclusion that the local holonomy map $f$ associated to a small loop around
$s_2$ and contained in $C_{k+1}$ has order equal to~$k+1$. Since $C_{k+1} \setminus \{ s_0, s_1, s_2\}$ is a regular leaf
of $\tilf$, we conclude that the (image of the) holonomy representation of the fundamental group of
$C_{k+1} \setminus \{ s_0, s_1, s_2\}$ in $\diff$ is nothing but the group generated by $f,g$.
Note that this conclusion depends only on the configuration of the reduction tree. In turn
this configuration depends only on a jet of finite order of $\om$. Therefore every perturbation of the coefficients
of $\om$ affecting only those of sufficiently high order will give rise to
perturbations of the above indicated generators preserving their
fixed (finite) orders. Since every local diffeomorphism of finite order is conjugate to the corresponding rotation,
it follows that the mentioned perturbations are made inside the conjugacy classes of $f$ and $g$.
This also justifies the fact that in Theorems~A and~B, only perturbations of local diffeomorphisms that do not alter
the corresponding conjugation classes were considered.

Conversely given two local diffeomorphism $\tilde{f}, \tilde{g}$ of orders respectively $2, \, k+1$, these diffeomorphism
can be realized (up to simultaneous conjugation) as the holonomy of the corresponding component $C_{k+1}$ for some local foliation
$\fol_{\om}$ (or $\tilf$). This is done through a well-known gluing procedure for which precise references will be provided later.
Therefore the set of all foliations
$\fol_{\om}$, up to conjugation, can also be parameterized by the pair of elements $\tilde{f}$ an
 $h^{-1} \circ tilde{g} \circ h$ for some $h \in \diff$.

We can now state a sharper version of Theorem~C.

\begin{teo}\label{foliation}
Let $\om \in \Lambda_{(\C^2,0)}$ be a $1$-form with an isolated singularity at the origin and defining
a germ of a nilpotent foliation $\fol$ of type $A^{2k+1}$. Then for each $N \in \N$ there exists a $1$-form
$\om^{\prime} \in \Lambda_{(\C^2,0)}$ defining a germ of a foliation $\fol^{\prime}$ and satisfying the
following conditions:
\begin{itemize}
\item[(a)] $J_0^N \om^{\prime} = J_0^N \om$.

\item[(b)] $\fol$ and $\fol^{\prime}$ have $S$ as a common separatrix.

\item[(c)] there exists a fundamental system of open neighborhoods $\{U_n\}_{n \in \N}$ of $S$, inside a
closed ball $\bar{B}(0,R)$, such that for all $n \in \N$
\begin{itemize}
\item[(c1)] The leaves of the restriction of $\fol^{\prime}$ to $U_n \setminus S$, $\fol^{\prime}|_{(U_n \setminus S)}$
are simply connected except for a countable number of them.

\item[(c2)] all leaves of $\fol^{\prime}|_{(U_n \setminus S)}$ are incompressible, i.e. their fundamental groups inject in the
fundamental group of $U_n \setminus S$.

\item[(c3)] the morphism $\Pi_1(U_n \setminus S, .) \rightarrow \Pi_1(\bar{B}(0,R) \setminus S, .)$ induced by
the inclusion map is an isomorphism.
\end{itemize}
\end{itemize}
\end{teo}

To begin with we are going to describe the construction of the system of neighborhoods $\{ U_n\}$. In fact,
we are going to construct a single neighborhood, denoted by $U$, that can be thought of as being $U_1$. The
remaining neighborhoods are obtained by the same procedure. Let us then fix a tubular neighborhood $\mathcal{T}$
of $C_{k+1}$ equipped with a locally trivial ($C^{\infty}$) fibration $\xi : \mathcal{T} \mapsto C_{k+1}$
whose fibers are discs and such that $C_k, \, C_{k+2}$ and the transform of $S$ are all contained in fibers
of $\xi$.

Given a path $\alpha : [0,1] \rightarrow C_{k+1}$, let $\vert \alpha \vert$ denote its image (or trace) in
$C_{k+1}$. Concerning the component $C_{k+1}$, we define the following:
\begin{itemize}
  \item Three simple loops $\delta_0, \, \delta_1, \, \delta_2$ about $s_0, s_1, s_2$ defining three closed
  discs $\overline{D}_0, \overline{D}_1, \overline{D}_2$ contained in $C_{k+1}$ and such that $\overline{D}_i \cap
  \overline{D}_j =\emptyset$ for $i \neq j$.
  \item Two disjoint simple paths $\sigma_1, \sigma_2$ in $C_{k+1} \setminus ( D_0 \cup D_1 \cup D_2)$
  such that $\sigma_1 (0), \sigma_2 (0)$ lie in $\delta_0$ and $\sigma_1 (1) = \delta_1 (0)$,
  $\sigma_2 (1) = \delta_2 (0)$. Here $D_i$ stands for the interior of $\overline{D}_i$.
  \item A base point $\tilde{s}$ in the simply connected open set
  $$
  C_{k+1}^{\ast} = C_{k+1} \setminus (\overline{D}_0 \cup \overline{D}_1 \cup \overline{D}_2
  \cup \vert \sigma_1 \vert \cup \vert \sigma_2 \vert) \, .
  $$
  \item A conformal disc $\Omega$ in $\xi^{-1} (\tilde{s})$ with center identified to $\tilde{s}$.
\end{itemize}
Figure~\ref{Figure2} summarizes these definitions. Modulo choosing $\Omega$ sufficiently small, its saturated
by the restriction of $\tilf$ to $\xi^{-1} (C_{k+1}^{\ast})$ is homeomorphic to the product $\Omega \times
C_{k+1}^{\ast}$ endowed with the horizontal foliation. Set $V_0$ to be the closure of
$\xi^{-1} (C_{k+1}^{\ast})$ in $\xi^{-1} (C_{k+1} \setminus (\overline{D}_0 \cup \overline{D}_1 \cup \overline{D}_2))$.

\begin{figure}[hbtp]
\centering
\includegraphics[scale=0.8]{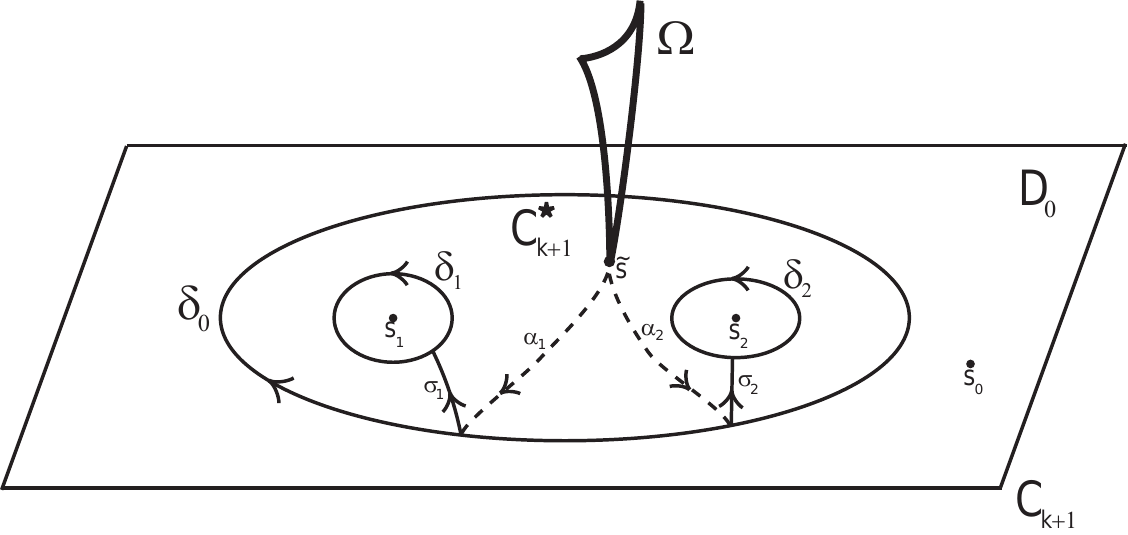}
\caption{}
\label{Figure2}
\end{figure}

The foliation induced on $\xi^{-1} (\delta_i)$, $i=1,2$, is the suspension over $\delta_i$ of the holonomy map
of $\tilf$ associated to the loop $\delta_i$ and realized at a transverse section $\Omega_i = \xi^{-1} (\delta_i (0))$.
Next fix two simple paths $\alpha_i : [0,1] \rightarrow C_{k+1}$, $i=1,2$, such that $\alpha_1 (0) = \alpha_2 (0) = \tilde{s}$
and $\alpha_i (1) = \sigma_i (0)$, cf. Figure~\ref{Figure2}. Set
$$
\gamma_i = \alpha_i \sigma_i \delta_i \sigma_i^{-1} \alpha_i^{-1} \; , \; \; i=1,2 \;\;\, {\rm and} \; \; \,
\Delta = \vert \gamma_1 \vert \cup \vert \gamma_2 \vert \, .
$$
Using a suitably chosen (real) vector field, we can construct a deformation retraction from
$C_{k+1} \setminus (\overline{D}_0 \cup \overline{D}_1 \cup \overline{D}_2)$ to $\Delta$. This deformation can naturally
be lifted to a retraction from $V_0$ to $V_0 \cap \xi^{-1} (\Delta)$ obtained through leafwise homotopy. Therefore we have:

\begin{lema}
\label{applicationlemma1}
The following holds:
\begin{enumerate}
\item Every loop with base point in $\Omega$ contained in a leaf of the restriction of $\tilf$ to $V_0$
is homotopic inside the same leaf to a loop contained in $\xi^{-1} (\Delta)$.

\item $\xi^{-1} (\Delta)$ is a real $3$-manifold with boundary and corners and the restriction $\tilf_{\vert \xi^{-1} (\Delta)}$
of $\tilf$ to $\xi^{-1} (\Delta)$ is a singular real foliation of (real) dimension~$1$. Every loop $c$, with base
point in $\Omega$, contained in a leaf of this real foliation projects onto a loop contained in $\Delta=\vert \gamma_1 \vert \cup \vert
\gamma_2 \vert$, with base point $\tilde{s}$, defining (through homotopy) a word
$$
W(f,g) = F_{1} \circ  \cdots \circ F_r, \, \; \; F_i \in \{ f^{\pm j}, g^{\pm j} \},
$$
in the pseudogroup generated by $f,g$ on $\Omega$ (with the appropriate identifications).\qed
\end{enumerate}
\end{lema}

In fact, concerning item~(2) above, it is clear that the base point $c(0)$ of the loop $c$ belongs to some connected
component of the domain of definition of the element $W(f,g) = F_{1} \circ  \cdots \circ F_r$, where $F_i$ is as above.
Here $W(f,g) = F_{1} \circ  \cdots \circ F_r$ is viewed
as belonging to the pseudogroup generated on $\Omega$ by $f,g$ again with the natural identifications.

So far we have described $V_0$ that accounts for a neighborhood of $C_{k+1}^{\ast}$. The next step is to define a neighborhood
$V_1$ of $C_{k+2}$ and a neighborhood $V_2$ of the divisor constituted by the string of rational curves
going from $C_1$ to $C_k$ so as to satisfy the conditions of the lemma below.

\begin{lema}
\label{applicationlemma2}
There exists neighborhoods $V_1$ and $V_2$ as above such that $V = V_0 \cup V_1 \cup V_2$ fulfils
the following conditions:
\begin{description}
  \item[a] The union of $V$ with a neighborhood of $s_0$ provides a neighborhood of the (total) exceptional divisor.
   Besides this neighborhood coincides with the saturated of $\Omega$ by $\tilf$.
  \item[b] Every loop with base point in $\Omega$ and contained in a leaf $L$ of $\tilf_{\vert V}$ (the restriction of
  $\tilf$ to $V$) is homotopic in $L$ to a loop contained in $V_0$.
  \item[c] Every loop $\Lambda$ contained in a leaf of $\tilf_{\vert \xi^{-1} (\Delta)}$ projecting onto a loop that is a power
  of either $\gamma_1^2$ or $\gamma_2^{k+1}$ in homotopically trivial in the leaf of $\tilf_V$ containing $\Lambda$.
\end{description}
\end{lema}

\begin{proof}
To construct $V_1$ (resp. $V_2$) first note that every connected component of the intersection between a leaf of
$\tilf_{\vert V_0}$ and $\xi^{-1} (\vert \delta_1 \vert)$ (resp. $\xi^{-1} (\vert \delta_2 \vert)$) is either a segment
(i.e. a simple open path) or a circle. In the latter case, this circle bounds a disc in the corresponding leaf. In the case of
$\delta_1$ this is particularly immediate since $s_1$ is the unique singularity of $\tilf$ in $C_{k+2}$ so that the holonomy
associated to its regular part is trivial. A similar argument applies to $\delta_2$ and to the component $C_1$. As already
mentioned this forces all singularities of $\tilf$ lying in some of the components $C_1 ,\ldots ,C_k$ to be linearizable.
In particular the restriction of $\tilf$ to some neighborhood $V_{k+2}'$ (resp. $V_{0,k}'$) of $C_{k+2}$ (resp. of the divisor
consisting of the components $C_1, \ldots , C_k$) possesses a non-constant holomorphic first integral $F^{k+2}$
(resp. $F^{0,k}$). Indeed, $F^{k+2}$ maps $V_{k+2}'$ to a neighborhood $B_{k+2}$ of $0\in \C$ sending $C_{k+2}$ to $0 \in \C$.
Similarly $F^{0,k}$ maps $V_{0,k}'$ to a neighborhood $B_{0,k}$ of $0\in \C$ and takes the divisor consisting of the components
$C_0, \ldots , C_k$ to $0 \in \C$. Thus $F^{k+2}$ (resp. $F^{0,k}$) defines a locally trivial fibration over the
corresponding punctured neighborhood of $0 \in \C$. Furthermore it is immediate to check that the fibers of both fibrations
are discs. This establishes the claim.

Now we shall define $V_1$ (resp. $V_2$) as the union of all the discs bounded by the above mentioned circles.
In fact, we set $V_1  = (F^{k+2})^{-1} (K_{k+2})$ where
$$
K_{k+2} = \{ z \in B_{k+2} \setminus \{0 \} \ \, ; \, \; (F^{k+2})^{-1} (z) \cap \xi^{-1} (\vert \delta_1 \vert)
\; \; \, {\rm is \; \;  a \; \; circle}  \; \} \cup  \{ 0\} \, .
$$
Similarly $V_2 =(F^{0,k})^{-1} (K_{0,k})$ where
$$
K_{0,k} = \{ z \in B_{0,k} \setminus \{0 \} \ \, ; \, \; (F^{0,k})^{-1} (z) \cap \xi^{-1} (\vert \delta_2 \vert)
\; \; \, {\rm is  \; \;  a \; \;  circle}  \; \} \cup  \{ 0\} \, .
$$
It is now clear that $V = V_0 \cup V_1 \cup V_2$ satisfies the conditions (a), (b) and (c) in the statement.
\end{proof}

\begin{proof}[Proof of Theorem~\ref{foliation}]
Consider a reduced non-trivial word $W (f,g)$ in $f,g$ and their inverses. Here is regarded as an element
of the pseudogroup generated by $f,g$ on $\Omega$. Modulo reducing $\Omega$ we suppose that, if $W(f,g)$ as
above coincides with the identity on some connected component of its domain of definition (itself contained in
$\Omega$), then $W (f,g)$ can be written as a word in $g^2$, $f^{k+1}$ and their inverses. According to Theorem~B,
this can always be achieved by conjugating $g$ by some element $h$. Furthermore $h$ can also be supposed arbitrarily
close to the identity in the analytic topology and, in view of Theorem~\ref{lastversion3.3}, $h$ can also be supposed
tangent to the identity to a given fixed order $N$.

Next consider a loop $\Lambda$ in a leaf $L$ of $\tilf_{\vert V}$ with $V$ as in Lemma~\ref{applicationlemma2}. The item~(a) of
this lemma allows us to assume that $\Lambda (0) =\Lambda (1) \in \Omega$. Furthermore, by virtue of item~1 of
Lemma~\ref{applicationlemma1} and of item~(b) of Lemma~\ref{applicationlemma2}, $\Lambda$ is homotopic in $L$
to a loop $\Lambda'$ contained in $\xi^{-1} (\Delta)$. Now we have:
\begin{itemize}
  \item  If $\xi \circ \Lambda'$ is homotopic to an element (loop) contained in the subgroup of the fundamental group
  of $C_{k+1} \setminus \{ s_0, s_1 ,s_2 \}$ generated by $\gamma_1^2, \gamma_2^{k+1}$. Then
  item~(c) of Lemma~\ref{applicationlemma2} guarantees that $\Lambda'$, and thus $\Lambda$ itself, is homotopically trivial in $L$.

  \item If $\xi \circ \Lambda'$ is not homotopic to an element (loop) contained in the subgroup of the fundamental group
  of $C_{k+1} \setminus \{ s_0, s_1 ,s_2 \}$ generated by $\gamma_1^2, \gamma_2^{k+1}$. Then the holonomy map associated
  to $\Lambda'$ does not coincide with the identity on any connected component of its domain of definition in $\Omega$.
  Therefore the leaves that are not simply connected intersect $\Omega$ at that, necessarily isolated, fixed points by
  some non-trivial element of the pseudogroup generated on $\Omega$ by $f,g$. This set however is clearly countable.
\end{itemize}

The preceding discussion accounts for almost all of the statement of Theorem~\ref{foliation}. It remains only to verify
item~(a). This means that we are given a $1$-form $\omega$ representing a foliation $\fol$ with singularity of type
$A^{2k+1}$. As it was seen the fact that $\fol$ belongs to the class $A^{2k+1}$ is determined by the reduction of the separatrix
and hence by a finite jet of $\omega$. Fix $N \in \N$ large enough to imply that every $1$-form $\om^{\prime}$
that is tangent to $\om$ at $0 \in \C$ to order $N$ automatically defines a foliation $\fol^{\prime}$ in $A^{2k+1}$. We need
to show that we can obtain $\fol^{\prime}$ leading to a foliation $\tilf^{\prime}$ whose holonomy group associated to
the component $C_{k+1}$ satisfies the generic conditions of Theorem~B. To obtain $\om^{\prime}$ it suffices to construct an
``equireducible'' deformation of $\fol$ whose holonomy groups associated to the component $C_{k+1}$ is as in
Theorem~\ref{lastversion3.3} i.e. generic in the sense of Theorem~B and obtained from the corresponding group of the initial
$\fol$ by conjugating ``$g$'' by a local diffeomorphism $h$ tangent to the identity at $0\in \C$ to order $M$. Once we have
chosen $h$ as indicated, the construction of $\om^{\prime}, \, \fol^{\prime}$ is carried out by resuming word-by-word the
constructions presented in \cite{lefloch} or in \cite{matteisalem} which, in turn, are based on the ``th\'eor\`eme de stabilit\'e~1.4.6''
from \cite{matteisalem}. This ends the proof of Theorem~\ref{foliation}.
\end{proof}

\bigskip

\begin{flushleft}
{\sc Jean-Fran\c{c}ois Mattei \,  \&  \, Julio Rebelo} \\
Institut de Math\'ematiques de Toulouse\\
118 Route de Narbonne\\
F-31062 Toulouse, FRANCE.\\
mattei@math.univ-toulouse.fr \\
rebelo@math.univ-toulouse.fr

\end{flushleft}

\bigskip

\begin{flushleft}
{\sc Helena Reis} \\
Centro de Matem\'atica da Universidade do Porto, \\
Faculdade de Economia da Universidade do Porto, \\
Portugal\\
hreis@fep.up.pt \\

\end{flushleft}

\end{document}